\documentclass[12pt]{amsart}
\usepackage[numbers, square]{natbib}
\usepackage{amssymb}
\usepackage{amsmath}
\usepackage{amsfonts}
\usepackage{geometry}
\usepackage{bbm}
\usepackage{dsfont}
\usepackage{graphicx}
\usepackage{amsthm}
\usepackage{hyperref}
\DeclareMathOperator\supp{supp}

\usepackage[dvipsnames]{xcolor}
\usepackage{esint}

\providecommand{\U}[1]{\protect\rule{.1in}{.1in}}
\newtheorem{theorem}{Theorem}[section]
\newtheorem{lemma}[theorem]{Lemma}
\newtheorem{corollary}[theorem]{Corollary}
\newtheorem{proposition}[theorem]{Proposition}

\newtheorem{remark}[theorem]{Remark}

\numberwithin{equation}{section}

\newcommand{\mbb}{\mathbb}
\newcommand{\mc}{\mathcal}

\newcommand{\ball}[2]{\mathbb{B}_{{#1}}({#2})}
\newcommand{\cball}[1]{\mathbb{B}_{{#1}}}
\newcommand{\hball}[2]{\mathbb{B}^{+}_{{#1}}({#2})}
\newcommand{\chball}[1]{\mathbb{B}^{+}_{{#1}}}
\newcommand{\lball}[1]{\mathbf{B}_{{#1}}}

\begin{document}

\title[]{Boundary quantitative unique continuation for solutions of elliptic equations}
\author{Jack Dalberg and Jiuyi Zhu}
\address{Department of Mathematics\\
Louisiana State University\\
Baton Rouge, LA 70803, USA\\
Email: jdalbe1@lsu.edu, zhu@math.lsu.edu }
\begin{abstract}
    We study the quantitative unique continuation on the boundary for solutions of elliptic equations with Neumann boundary conditions for bounded potentials and boundary potentials on compact manifolds with boundary. The boundary doubling inequality is derived from  the combination of  local Carleman estimates and global Carleman estimates. Some special attentions are paid to overcome the regularity issues arising from this boundary value problem.
\end{abstract}
\maketitle
\section{Introduction}
Let  $(\mathcal{M},g)$ be a smooth compact $n$-dimensional Riemannian manifold with boundary.  We consider solutions to the following elliptic equations with the Neumann boundary conditions
\begin{equation}\label{bound PDE}
     \begin{cases} 
      \Delta_{g} u = H(x)u, \quad & x\in \mathcal{M}, \\
      \frac{\partial u}{\partial \nu} = h(x) u, & x\in \partial\mathcal{M},
   \end{cases}
\end{equation}
where $\Delta_{g}=\sum\limits_{i,j=n}g^{-\frac{1}{2}} \partial_i(g^{\frac{1}{2}} g^{ij}\frac{\partial }{\partial x_j})$ is the Laplace-Beltrami operator, $\nu$ is a unit outer normal, $H \in L^{\infty}(\mc{M})$ and $h \in L^{\infty}(\partial \mc{M})$. The aim of this paper is to obtain a quantitative boundary doubling inequality for solutions to \eqref{bound PDE}, which implies the vanishing order of solutions on the boundary $\partial \mc{M}$. The vanishing order of a function at some point is defined to be the highest order derivative such that all lower order derivatives at the point vanish. The doubling inequality and vanishing order describe the quantitative behavior of strong unique continuation property. Strong unique continuation property states that 
a solution vanishes globally 
if the solution  vanishes of infinite order at a
point. See e.g. \cite{JK85}, \cite{KT01} and references therein. We know that all zeros of nontrivial solutions of second order linear
equations with bounded potentials and boundary potentials on smooth compact Riemannian manifolds or its boundary are of finite order.
If the strong unique continuation property holds for the solutions and solutions are not trivial, then the vanishing order of solutions depends on the potential functions and
coefficient functions appearing in the equations. It is interesting to characterize the vanishing
order by the potential functions. We call the research on the doubling inequalities and vanishing order as the study of quantitative unique continuation.

There is an extensive literature on the study of quantitative unique continuation  for solutions of elliptic equations in the interior of manifold. Let us first briefly review some literature. Let $\phi_\lambda$ be  the Laplace eigenfunction  on the boundaryless compact manifold $\tilde{\mathcal{M}}$,
\begin{align}
 -\Delta_{g} \phi_\lambda = \lambda\phi_\lambda, \quad & x\in \tilde{\mathcal{M}}.
 \label{eigen-1}
\end{align}
Donnelly and Fefferman in the seminal paper \cite{DF88} obtained the following celebrated doubling inequality 
\begin{align}
\|\phi_\lambda\|_{L^{2}(\mathbb B_{2r}(x))}\leq e^{C\sqrt{\lambda}} \|\phi_\lambda\|_{L^{2}(\mathbb B_{r}(x))}
\label{doub-df}
\end{align}
for any $x\in \tilde{\mathcal{M}}$, where the square root of $\lambda$ in (\ref{doub-df}) is sharp and $C$ depends on the $\tilde{\mathcal{M}}$. This sharpness of $\sqrt{\lambda}$ in (\ref{doub-df}) can be seen from spherical harmonics if $\tilde{\mathcal{M}}=\mathbb S^{n-1}$. The doubling inequality (\ref{doub-df}) implies that  the maximal vanishing order of $\phi_\lambda$ is at most $C\sqrt{\lambda}$. The doubling inequality  (\ref{doub-df}) also plays a very important role in the study of the measure of nodal sets for eigenfunctions $\phi_\lambda$, see e.g. \cite{DF88}, \cite{Lin91}, \cite{L18}.

To study the role of potential functions in quantitative unique continuation, Kukavica initiated the study of  the vanishing order of solutions for Schr\"odinger equation in \cite{K98},
\begin{align}
    \triangle u=H(x)u.
    \label{schro-1}
\end{align}
If $H(x)\in C^1$, the upper bound of vanishing order is shown to be less than $C(1+\|H\|_{C^1})$. Recently, the sharp
vanishing order for solutions of (\ref{schro-1}) was shown to be less than $C(1+ \|H\|^{\frac{1}{2}}_{C^1})$  independently in \cite{B12} and \cite{Z16} by different methods.
If $H(x)\in L^\infty$, the phenomenon of quantitative unique continuation becomes delicate. Bourgain and Kenig \cite{BK05}  studied the model (\ref{schro-1}) motivated
by the work on Anderson localization for the Bernoulli model. Bourgain and Kenig established
 that the order of vanishing for
solutions is less than $C(1+ \|H\|^{\frac{2}{3}}_{L^\infty})$. Especially, Kenig in \cite{K07} also pointed out that the exponent $\frac{2}{3}$ power in the bound $C(1+ \|H\|^{\frac{2}{3}}_{L^\infty})$ is sharp for complex-valued $H(x)$ based on Meshkov’s example in \cite{M92}.
Davey in \cite{D14} and Bakri in \cite{B13} generalized the quantitative unique continuation result of solutions to more
general elliptic equations of the form 
\begin{align}\label{HHH}
    -\triangle u+H_1(x)\cdot \nabla u+ H(x)u=0
\end{align} 
for possibly 
complex-valued bounded potential functions $H_1(x)$  and $H(x)$. It was shown that
the vanishing order for the solutions is less than $ C(1 + \|H_1\|^2_{L^\infty}+ \|H\|^\frac{2}{3}_{L^\infty} )$. The square power in the term $\|H_1\|^2_{L^\infty}$ was shown be sharp for complex-valued $H_1$ in \cite{D14}.

Based on Donnelly and Fefferman’s work on the vanishing order of eigenfunctions in (\ref{doub-df}), it was asked if the order of vanishing can be reduced to $C(1 + \|H\|_{L^\infty}^{\frac{1}{2}})$ for real-valued $u$ and
$H$ for the solutions in  (\ref{schro-1}) in \cite{BK05}. It is equivalent to a quantitative form of Landis’ conjecture in the
real-valued setting. In the late 1960s, E.M. Landis conjectured that the bounded solution $u$
to $\triangle u- H(x)u = 0$ in $\mathbb R^n$
is trivial if $|u|\leq e^{-C|x|^{1+\varepsilon}}$ for any $\varepsilon>0$, where $H(x)$ is a bounded function. The Landis' conjecture was settled recently in \cite{LMNN20} in the plane. See also \cite{KSW15} for the answer of this conjecture in $\mathbb R^2$ for nonnegative real-valued potential $H$.

As the strong unique continuation property holds on $\partial\mathcal{M}$ for the solutions in (\ref{bound PDE}), we are interested in the role of $H(x)$ and $h(x)$ on the vanishing order of $u$ on $\partial\mathcal{M}$. If $H(x)\in C^1$ and $h(x)\in C^1$, following the arguments in  \cite{Z21}, we could show the doubling inequality 
\begin{align}\label{differe-1}
     \|u\|_{L^{2}(\lball{2r}(x))} \leq e^{C(1+\|H\|^{\frac{1}{2}}_{C^1} + \|h\|_{C^1})}\|u\|_{L^{2}(\lball{r}(x))} 
\end{align}
for  any $\lball{2r}(x) \subset \partial\mathcal{M}$. Due to the delicate role of the $L^\infty$ norm that plays in the quantitative unique continuation for  the Laplace operator in (\ref{schro-1}), we aim to explore the role of $h(x)$ on the boundary quantitative unique continuation for (\ref{bound PDE}).

As the study of  (\ref{schro-1}) is motivated by the study of Laplace eigenvalue problem (\ref{eigen-1}), the study of the Neumann boundary problem is inspired by the study of Steklov eigenvalue problem
\begin{equation}\label{stek}
     \begin{cases} 
      \Delta_{g} \phi_\lambda = 0, \quad & x\in \mathcal{M}, \\
      \frac{\partial \phi_\lambda}{\partial \nu} = \lambda\phi_\lambda, & x\in \partial\mathcal{M}.
   \end{cases}
\end{equation}
The vanishing order and doubling inequalities have been studied in e.g. \cite{BL15}, \cite{Z15}, \cite{R17}. The boundary doubling inequality for Steklov eigenfunction was shown as
\begin{align}
     \|\phi_\lambda\|_{L^{2}(\lball{2r}(x))} \leq e^{C\lambda}\|\phi_\lambda\|_{L^{2}(\lball{r}(x))},
\end{align}
which implies that $\phi_\lambda$ vanishes at most $C\lambda$ on $\partial\mathcal{M}$.

Generally speaking, two principle methods are used to study  quantitative unique continuation property. One method is the Carleman estimates. Carleman estimates are weighted integral estimates, see e.g. \cite{DF88},\cite{K07}, and \cite{Z15}. The second method is called frequency functions which are given in the term of quotient. The  frequency function has deep geometric perspectives and describes the degree of leading polynomials if the function is locally expanded as the polynomial functions, see e.g. \cite{GL86},\cite{K98},\cite{Z16}. Both methods fail to distinguish between real-valued functions and complex-valued functions. We will apply Carleman estimates to obtain the quantitative unique continuation results. Hence, the potential function $H$ and $h$ in (\ref{bound PDE}) can be complex-valued functions.  There is no effective way to control the boundary term with the $L^\infty$ norm in the Carleman estimates. To overcome the difficulty and obtain the desirable Carleman estimates, we split the equations into the sum of two elliptic equations. We obtain two types of Carleman estimates for these two equations and use the one with  strong Carleman parameter $\tau$ to absorb the boundary term. To obtain the boundary doubling inequality, we also need the $H^2$ estimates for some interpolation arguments. See (\ref{interp ineq}) in Section 3. However, based on the assumption of $H(x)$ and $h(x)$ in (\ref{bound PDE}), solutions $u$ are at most in $H^{\frac{3}{2}, 2}$ or $W^{1, p}$ with $1<p<\infty$. Due to the lack of $H^2$ regularity, the novelty is that we employ a lifting argument and reduce the equation (\ref{bound PDE}) to some new elliptic equation with zero Neumann boundary conditions. Hence, we are able to obtain the quantitative $H^2$ estimates for the solutions of the new elliptic equations. We establish the following quantitative doubling inequality on the boundary for the solutions in (\ref{bound PDE}).



\begin{theorem}\label{thm 1}
Let $u$ be a solution to \eqref{bound PDE}. There exist a positive constant $C$ depending only on the manifold $\mc{M}$ such that
\begin{equation}
\label{dou-th}
    \|u\|_{L^{\infty}(\lball{2r}(x))} \leq e^{C(1+\|H\|^{\frac{2}{3}}_{L^{\infty}} + \|h\|^{2}_{L^{\infty}})}\|u\|_{L^{\infty}(\lball{r}(x))}
\end{equation}
for  any $\lball{2r}(x) \subset \partial \mc{M}$.
\end{theorem}

An immediate corollary for the theorem \ref{thm 1} is the vanishing order of solution in (\ref{bound PDE})
\begin{corollary}
    The solution of \eqref{bound PDE} vanished at most of order $C(1+\|H\|^{\frac{2}{3}}_{L^{\infty}} + \|h\|^{2}_{L^{\infty}})$ on $\partial \mc{M}$, where $C$ depends only on the manifold $\mc{M}$ 
\end{corollary}
The corollary can be shown by iteration of the doubling inequality in Theorem \ref{thm 1}.  We want to make some comments on the possible sharp exponents in the doubling inequalities in Theorem \ref{thm 1}.

\begin{remark}
   The power $\frac{2}{3}$ on $\|H\|_{L^{\infty}}$ is sharp as mentioned before. The power $2$ on $\|h\|_{L^{\infty}}$ arises from the local quantitative Carleman estimates in Proposition \ref{local car}. Because of the lack of differentiability of $h(x)$, we can not do further integration by parts on the boundary and bring down the power for the norm of $h(x)$ as the doubling inequality (\ref{differe-1}) in \cite{Z21}. Another insight of the power $2$ on $\|h\|_{L^{\infty}}$ comes from the lifting arguments in section 3. The new solution $\bar u$  in (\ref{bar-u}) and $u$ have the same vanishing order on the boundary. We can compare the $\nabla w$ in (\ref{bar-u}) with $H_1(x)$ in (\ref{HHH}). The power $2$ on $\|H_1\|^2_{L^{\infty}}$ is the sharp power for the vanishing order or doubling inequality of (\ref{HHH}) as shown in \cite{D14}. In (\ref{infty-1}), the norm of  $\nabla w$ is bounded by $\|h\|_{L^{\infty}}$, which sheds some light on the observation that the term  $\|h\|^{2}_{L^{\infty}}$ in (\ref{dou-th}) is possibly sharp.
\end{remark}

The paper is organized as follows. In Section 2, we state a local Carleman estimate in the half balls, then we apply the local Carleman estimates to obtain  three half-ball inequalities and doubling inequalities in  the half balls.  Section 3 is devoted to  a global Carleman estimate, a quantitative Cauchy uniqueness result, and the proof of  Theorem \ref{thm 1}. In Section 4, we give the proof of  the local and global Carleman estimates. In the appendix, we include the proof of a quantitative Caccioppoli inequality for (\ref{bound PDE}) and a regularity result. Throughout the paper, the letters $C$, $\hat{C}$ or $C_i$ denote positive constants that do not depend on $V$ or the solution $u$, and may vary from line to line.

\textbf{Acknowledgement.} Both of authors are partially supported by NSF DMS-2154506. 

\section{Doubling inequalities in half-balls }
Let us first introduce the Carleman estimates in the half balls. We will make use of Fermi coordinate system near the boundary $\partial\mathcal{M}$. For any point $p\in \partial\mathcal{M}$,  the Fermi exponential map at $p$ which gives the Fermi coordinate system, is defined in a half ball of $\mathbb R^n_+\approx T_p(\mathcal{M})$.
Assume $(x_1, x_2, \cdots, x_{n-1})$ be the geodesic normal coordinates of $\partial\mathcal{M}$ at $p$. Let $x_n=dist(x, \partial\mathcal{M})$. Note that $dist(x, \partial\mathcal{M})$ is smooth in a small open neighborhood of $p$. Thus, we can identify $\partial\mathcal{M}$ locally as $x_n=0$. By Fermi normal coordinate coordinates, 
\begin{align}
    \triangle_g u= g^{-\frac{1}{2}} \partial_i(g^{\frac{1}{2}} g^{ij}\frac{\partial u}{\partial x_j})
\end{align}
with $g^{nn}=1, g^{in}=0$ and $g^{ij}(x', x_n)\not=0$ for $1\leq i, j\leq n-1,$ where the summation of the index $i,j$ is understood by Einstein notations.
We also write $\triangle_g$ as $\triangle$ for convenience.
The Fermi normal coordinate system at $p$ facilitates the polar geodesic coordinates $(r,\theta)$ in the half ball $\mathbb{B}_{r_{0}}^{+}$, where $r_0$ is always less than the injectiviy radius of the manifold $\mathcal{M}$.

Let $r = r(y)$ be the Riemannian distance from the origin to $y$. Note that $r$ appeared in the integration denotes the distance function, and $r$ is also used as the radius for the ball $\mathbb B_r$.
We construct the Carleman weight function $\psi$ as follows. Let $\psi(y) = -\phi(\ln r(y))$, where $\phi(t) = t + \ln t^{2}$ for $(-\infty,T_{0}]$ and $T_{0}$ is negative with $|T_{0}|$ sufficiently large enough.  Then, we see that $\phi(t)$ satisfies the following properties
\begin{equation}\label{phi-def}
\phi' = 1 + 2/t, \quad \phi'' = -2/t^{2}
\end{equation}
and
\begin{equation}
    1+\dfrac{2}{T_{0}} \leq \phi'(t) \leq 1, \quad  \lim\limits_{t \to -\infty} \dfrac{-\phi''}{e^{t}} = \infty.
\end{equation}

The following Carleman estimates in the half balls hold. Note that $\|\cdot\|_{\mathbb{B}^+_{r_0}}$ is denoted as $\|\cdot\|_{L^2(\mathbb{B}^+_{r_0})}$.
\begin{proposition}\label{local car}
Assume $H(x) \in L^{\infty}(\mathbb{B}_{r_{0}}^{+}), h(x) \in L^{\infty}(\partial \mathbb{B}_{r_{0}}^{+}\cap\{x|x_n=0\})$ and $F(x)\in L^2(\mathbb{B}_{r_{0}}^{+})$. Let $v \in H^1(\mathbb{B}_{r_{0}}^{+})$ be a solution of
\begin{equation}
    \begin{cases}
    \Delta_{g} v = H(x) v+F&  \hspace{0.5cm} in \hspace{0.1cm} \mathbb{B}_{r_{0}}^{+}, \\
       \frac{\partial v} {\partial {\nu}}  = h(x) v& \hspace{0.5cm} on \hspace{0.1cm} \partial\mathbb{B}_{r_{0}}^{+}\cap \{x \hspace{0.1cm} | \hspace{0.1cm} x_{n} = 0\}
    \end{cases}
\end{equation}
with $\supp v \subset \mathbb{B}_{r_{0}}^{+}\backslash \mathbb{B}_{\rho}^{+}$ for some $0<\rho<r_0$.
There exist positive constants $C_{0}, C_{1}$ such that for 
$$
\tau \geq C_{0}(1+\|H\|^{\frac{2}{3}}_{L^{\infty}}+\|h\|^{2}_{L^{\infty}}),
$$
one has
\begin{equation}\label{car est}
\begin{aligned}
\|r^{2}e^{\tau\psi}F\|_{\mathbb{B}_{r_{0}}^{+}}^{2} &\geq C_1\tau^{3}\|e^{\tau\psi} (\ln r)^{-1}v\|^{2}_{\mathbb{B}_{r_{0}}^{+}} 
       +C_1\tau \|e^{\tau\psi}(\ln r)^{-1}r\nabla v\|_{\mathbb{B}_{r_{0}}^{+}}^{2} \\&+C_1\tau^2 \rho\|e^{\tau\psi} r^{-1/2}v\|^{2}_{\mathbb{B}_{r_{0}}^{+}}.
    \end{aligned} 
\end{equation}
\end{proposition}
Since the proof of (\ref{car est}) is lengthy, we present its proof in Section 4. As the application of this local Carleman estimate, we can get the following quantitative three half-ball inequality.

\begin{lemma}\label{lemma THB}
 There exist positive constants $r_{0}$, $C$, and $0<\mu<1$ which depend only on $\mc{M}$ such that, for any $2r < r_{0}$ and any $x_{0} \in \partial \mc{M}$, the solution $u$ of \eqref{bound PDE} satisfies
\begin{equation}\label{three half balls}
    \|u\|_{L^{\infty}(\mbb{B}^{+}_{\frac{3r}{5}}(x_{0}))} \leq e^{C(\|H\|^{\frac{2}{3}}_{L^{\infty}} + \|h\|^{2}_{L^{\infty}})} \|u\|_{L^{\infty}(\mbb{B}^{+}_{2r}(x_{0}))}^{\mu}\|u\|_{L^{\infty}(\mbb{B}^{+}_{\frac{r}{2}}(x_{0}))}^{1-\mu}.
\end{equation}

\end{lemma} 
  Denote by $A_{R_{1}, R_{2}} = \{x \in \mc{M} | R_{1} \leq r(x) \leq R_{2}\}$. 

\begin{proof}[Proof of Lemma \ref{lemma THB}]
Without loss of generality, let $x_{0} = 0$ and $r_{0} = 1$.  Take $r>0$ such that $2r < r_{0} = 1$. 
Let $u$ be a solution of \eqref{bound PDE}. We introduce a radial smooth  cut-off function $\eta(|x|)\in C^{\infty}_{0}(\mbb{B}^{+}_{1})$ such that $0 < \eta < 1$ satisfies $\eta(|x|) = 1$ in $A_{\frac{r}{4},r}$, $\eta(|x|) = 0$ outside of $ A_{\frac{r}{8},\frac{3r}{2}} $, and $|\nabla^{\alpha} \eta(|x|) | \leq C r^{-|\alpha|}$ for some constant $C$, where $\alpha=(\alpha_1, \cdots, \alpha_n)$ is a multi-index. Since $u$ is the solution of \eqref{bound PDE}, then $w=u\eta$ is the weak solution of 
\begin{equation}\label{new-pde}
     \begin{cases} 
         \Delta w = 2\nabla u\cdot\nabla \eta+\Delta\eta u+H(x) w&  \hspace{0.5cm} in \hspace{0.1cm} \mathbb{B}_{r_{0}}^{+}, \\
        \frac{\partial w}{ \partial{\nu} } = h(x) w& \hspace{0.5cm} on \hspace{0.1cm} \partial\mathbb{B}_{r_{0}}^{+}\cap \{x \hspace{0.1cm} | \hspace{0.1cm} x_{n} = 0\}.
   \end{cases}
\end{equation}
Applying the Carleman estimate (\ref{car est}) to $w=\eta u$, we get
\begin{equation*}
\begin{aligned}
       \|r^{2}e^{\tau\psi}(2\nabla \eta\cdot \nabla u + \Delta\eta  u)\|_{\mathbb{B}_{1}^{+}}^{2} \geq C\tau^2 r\|e^{\tau\psi} r^{-1/2} \eta u\|^{2}_{\mathbb{B}_{1}^{+}}.\\
    \end{aligned} 
\end{equation*}
 Using the properties of $\eta$, we then get
\begin{equation*}
\begin{aligned}
       \|re^{\tau\psi}\nabla u\|_{A_{\frac{r}{8},\frac{r}{4}}} + \|r e^{\tau\psi} \nabla u\|_{A_{r,\frac{3r}{2}}} &+ \| e^{\tau\psi} u\|_{A_{\frac{r}{8},\frac{r}{4}}} + \|e^{\tau\psi}u\|_{A_{r,\frac{3r}{2}}} \\ &\geq C\tau\|e^{\tau\psi}  u\|_{A_{\frac{r}{4}, \frac{3r}{4}}}.
    \end{aligned} 
\end{equation*}
We then take a maximum and minimum of $\psi$ and $r$ over each annuli to get
\begin{equation*}
\begin{aligned}
       e^{\tau\psi_{1}}(\|\nabla u\|_{A_{\frac{r}{8},\frac{r}{4}}} + \|u\|_{A_{\frac{r}{8},\frac{r}{4}}}) + e^{\tau\psi_{2}}(\|\nabla u\|_{A_{r,\frac{3r}{2}}} + \|u\|_{A_{r,\frac{3r}{2}}}) \geq C\tau \|e^{\tau\psi_{3}}\|  u\|_{A_{\frac{r}{4},\frac{3r}{4}}},
    \end{aligned} 
\end{equation*}
where $\psi_{1} = \max_{A_{\frac{r}{8},\frac{r}{4}}} \psi$, $\psi_{2} = \max_{A_{r,\frac{3r}{2}}} \psi$, and $\psi_{3} = \min_{A_{\frac{r}{4},\frac{3r}{4}}} \psi$. Note that the function $\psi$ decreasing with respect to $r$, we have $\psi_{1} - \psi_{3} >0$ and $\psi_{2} - \psi_{3} < 0$. Applying the Caccioppoli inequality (\ref{corollary Cac}) in the appendix and enlarging each annulus to a ball gives
\begin{equation*}
\begin{aligned}
       C(1+\|H\|_{L^{\infty}}^{\frac{1}{2}}+\|h\|_{L^{\infty}})(e^{\tau(\psi_{1}-\psi_{3})}\|u\|_{\chball{\frac{r}{2}}} + e^{\tau(\psi_{2}-\psi_{3})}\|u\|_{\chball{2r}}) \geq \|  u\|_{\chball{\frac{3r}{4}}}, \\
    \end{aligned} 
\end{equation*}
where we have enlarged the right hand side by adding $\|u\|_{\mbb{B}_{\frac{r}{4}}^{+}}$ to both sides. We now want to incorporate the $\|u\|_{\mbb{B}_{2r}^{+}}$ term on to the right hand side. To this end, we choose $\tau$ such that
$$
C(1+\|H\|_{L^{\infty}}^{\frac{1}{2}}+\|h\|_{L^{\infty}})e^{\tau(\psi_{2}-\psi_{3})}\|u\|_{\chball{2r}} \leq \frac{1}{2}\| u\|_{\chball{\frac{3r}{4}}}, 
$$
which holds if
$$
\tau \geq \frac{1}{\psi_{3} - \psi_{2}} \ln{\frac{2C (1+\|H\|_{L^{\infty}}^{\frac{1}{2}}+\|h\|_{L^{\infty}})\|u\|_{\chball{2r}}}{\| u\|_{\chball{\frac{3r}{4}}}}}.
$$
From this, we get
\begin{equation}\label{thb 1}
\begin{aligned}
       C(1+\|H\|_{L^{\infty}}^{\frac{1}{2}}+\|h\|_{L^{\infty}})e^{\tau(\psi_{1}-\psi_{3})}\|u\|_{\chball{\frac{r}{2}}}\geq \|  u\|_{\chball{\frac{3r}{4}}}. \\
    \end{aligned} 
\end{equation}
Recall that we need $\tau > C(1+\|H\|^{\frac{2}{3}}_{L^{\infty}} + \|h\|_{L^{\infty}}^{2})$ in order to apply the Carleman estimate.  Thus, we choose
$$
\tau = C(1+\|H\|^{\frac{2}{3}}_{L^{\infty}} + \|h\|_{L^{\infty}}^{2}) + \frac{1}{\psi_{3} - \psi_{2}} \ln{\frac{2C (1+\|H\|_{L^{\infty}}^{\frac{1}{2}}+\|h\|_{L^{\infty}})\|u\|_{\chball{2r}}}{\| u\|_{\chball{\frac{3r}{4}}}}}.$$
Substituting this value of $\tau$ into \eqref{thb 1} gives
\begin{equation*}
\begin{aligned}
       e^{C(1+\|H\|_{L^{\infty}}^{\frac{2}{3}}+\|h\|_{L^{\infty}}^{2})}\|u\|_{\chball{\frac{r}{2}}}\|u\|_{\chball{2r}}^{\frac{\psi_{1}-\psi_{3}}{\psi_{3}-\psi_{2}}}\geq \|  u\|_{\chball{\frac{3r}{4}}}^{\frac{\psi_{1}-\psi_{2}}{\psi_{3}-\psi_{2}}}. \\
    \end{aligned} 
\end{equation*}
Set $\mu = \frac{\psi_{3}-\psi_{2}}{\psi_{1}-\psi_{2}}$. By the property of $\psi$, we see that $0<\mu<1$ is independent of $r$.
Raising both sides to the power of $\mu$, we get 
\begin{equation*}
\begin{aligned}
       e^{C(1+\|H\|_{L^{\infty}}^{\frac{2}{3}}+\|h\|_{L^{\infty}}^{2})}\|u\|_{\chball{\frac{r}{2}}}^{\mu}\|u\|_{\chball{2r}}^{1-\mu}\geq \|  u\|_{\chball{\frac{3r}{4}}}. \\
    \end{aligned} 
\end{equation*}
Since the $L^{\infty}$ and $L^{2}$ norms are comparable for elliptic equations (\ref{bound PDE}), we derived the three half-ball inequality.
\end{proof}

We will need a propagation of smallness argument, which shows that the $L^\infty$ growth of $u$ in any ball is bounded below by the $L^\infty$ norm of $u$ globally.
\begin{lemma}[Growth Lemma] \label{lemma Gro}
For any $0<2r<r_0$, there exists a constant $C_{r}$ depending on $\mc{M}$ and $r$ such that for any $x_{0} \in \partial \mc{M}$
$$
\|u\|_{L^{\infty}(\hball{r}{x_{0}})} \geq e^{-C_{r}(1+\|H\|_{L^{\infty}}^{\frac{2}{3}}+\|h\|_{L^{\infty}}^{2})} \|u\|_{L^{\infty}(\mc{M})}.$$
    
\end{lemma}

\begin{proof}
    Fix $0 < 2r <r_{0}$, with $r_{0}$ as in Lemma \ref{lemma THB}.  Take arbitrary $x_{0} \in \partial \mc{M}$. By rescaling, we may assume $\|u\|_{L^{\infty}(\mc{M})} = 1$.  By the regularity results of (\ref{bound PDE}), $u\in W^{1, p}$ for any $1<p<\infty$. Thus, $u\in C^{\beta}(\overline{\mathcal{M}})$ for $0<\beta<1$. By continuity, there exists a point $\hat{x} \in {\overline{\mathcal{M}}}$ such that $|u(\hat{x})| = \|u\|_{L^{\infty}(\mc{M})}$. The point $\hat{x}$ may lie on the boundary or in the interior of the manifold. We introduce a band around the boundary given by $\mc{M}_{r} = \{x\in \mc{M} \hspace{.2cm} | \hspace{.2cm} d(x,\partial \mc{M}) < r\}$ and  consider two cases.
    
   \textit{  Case 1:} Suppose $\hat{x} \in \overline{\mathcal{M}} \setminus \mc{M}_{r}$. Then we take a point $x_{1} \in \mc{M}\setminus \mc{M}_{\frac{2r}{3}}$ such that $\ball{\frac{r}{4}}{x_{1}} \subset \hball{r}{x_{0}}$.   
The following interior three-ball inequality holds, see e.g. \cite{B13},
\begin{equation}
    \|u\|_{L^{\infty}(\mbb{B}_{\frac{r}{4}}(x))} \leq e^{C(1+\|H\|^{\frac{2}{3}}_{L^{\infty}} )} \|u\|_{L^{\infty}(\mbb{B}_{\frac{r}{2}}(x))}^{\kappa}\|u\|_{L^{\infty}(\mbb{B}_{\frac{r}{8}}(x))}^{1-\kappa},
    \label{bark}
\end{equation}
where $0<\kappa<1$ and any $x\in \mc{M}\setminus \mc{M}_{\frac{2r}{3}}$.
     From application of the interior three-ball inequality (\ref{bark}),  we see that
\begin{equation*}
    \begin{aligned}
    \|u\|_{L^{\infty}(\ball{\frac{r}{4}}{x_{1}})} &\leq e^{C(1+\|H\|_{L^{\infty}}^{\frac{2}{3}})} \|u\|^{\kappa}_{L^{\infty}(\ball{\frac{r}{8}}{x_{1}})} \|u\|^{1-\kappa}_{L^{\infty}(\ball{\frac{r}{2}}{x_{1}})} \\
    &\leq  e^{C(1+\|H\|_{L^{\infty}}^{\frac{2}{3}})} \|u\|^{\kappa}_{L^{\infty}(\ball{\frac{r}{8}}{x_{1}})}  \\
    &\leq e^{C(1+\|H\|_{L^{\infty}}^{\frac{2}{3}})} \|u\|^{\kappa}_{L^{\infty}(\hball{r}{x_{0}})},
    \end{aligned}
 \end{equation*}
 where we used $\|u\|^{1-\kappa}_{L^{\infty}(\ball{\frac{r}{2}}{x_{1}})} \leq \|u\|_{L^{\infty}(\mc{M})}= 1$. We define a sequence of points in $\mc{M}$, say  $x_{2},...,x_{m} = \hat{x}$, such that $\ball{\frac{r}{8}}{x_{i+1}} \subset \ball{\frac{r}{4}}{x_{i}}$ for all $i= 1,...,m-1$. Thanks to (\ref{bark}), it follows that   $$
    \|u\|_{L^{\infty}(\ball{\frac{r}{8}}{x_{i+1}})} \leq e^{C(1+\|H\|_{L^{\infty}}^{\frac{2}{3}})} \|u\|^{\kappa}_{L^{\infty}(\ball{\frac{r}{8}}{x_{i}})}.
    $$
    The constant $m$ only depends on $diam(\mc{M})$ and $r$. Using the iteration of  interior three-ball inequality (\ref{bark}), then we see that
    $$
   1 = \|u\|_{L^{\infty}(\ball{\frac{r}{4}}{x_{m}})} \leq e^{C(m, \kappa) (1+\|H\|_{L^{\infty}}^{\frac{2}{3}})} \|u\|^{\kappa^m}_{L^{\infty}(\hball{r}{x_{0}})}.
    $$
    Hence 
    \begin{equation}
        \|u\|_{L^{\infty}(\hball{r}{x_{0}})} \geq e^{-C_{r}(1+\|H\|_{L^{\infty}}^{\frac{2}{3}})} \|u\|_{L^{\infty}(\mc{M})}.
        \label{kao-j}
    \end{equation}

    \textit{Case 2:} Suppose $\hat{x} \in \mc{M}_{r}$. Then there exists a point $\bar{x} \in \partial \mc{M}$ such that $\|u\|_{L^{\infty}(\hball{r}{\bar{x}})} = |u(\hat{x})|$. We define a sequence of points on $\partial \mc{M}$, say $\bar{x}_0, \bar{x}_{1},...,\bar{x}_{\bar{m}} = \bar{x}$ such that $\hball{\frac{r}{2}}{\bar{x}_{i+1}} \subset \hball{r}{\bar{x}_{i}}$ for all $i= 0,...,\bar{m}-1$ and $\bar{x}_0=x_0$. The constant $\bar{m}$ only depends on $\partial \mc{M}$ and $r$. We apply the three half-balls inequality \eqref{three half balls} under some rescaling to see that
    $$
    \begin{aligned}
    \|u\|_{L^{\infty}(\hball{\frac{r}{2}}{\bar{x}_{i+1}})} &\leq \|u\|_{L^{\infty}(\hball{r}{\bar{x}_{i}})} \\
    &\leq e^{C(\|H\|^{\frac{2}{3}}_{L^{\infty}} + \|h\|^{2}_{L^{\infty}}+1)} \|u\|_{L^{\infty}(\hball{2r}{\bar{x}_{i}})}^{1-\mu}\|u\|_{L^{\infty}(\hball{\frac{r}{2}}{\bar{x}_{i}})}^{\mu} \\
    &\leq e^{C(\|H\|^{\frac{2}{3}}_{L^{\infty}} + \|h\|^{2}_{L^{\infty}}+1)}\|u\|_{L^{\infty}(\hball{\frac{r}{2}}{\bar{x}_{i}})}^{\mu},
    \end{aligned}
    $$
    where we have used that $\|u\|_{L^{\infty}(\hball{2r}{\bar{x}_{i}})}^{1-\mu} \leq \|u\|_{L^{\infty}(\mc{M})}= 1$. We iterate this inequality to see that
    \begin{equation*}
        \begin{aligned}
1 &= \|u\|_{L^{\infty}(\hball{r}{\bar{x}_{\bar{m}}})} \\
    &\leq e^{C(\mu, \bar m) (\|H\|^{\frac{2}{3}}_{L^{\infty}} + \|h\|^{2}_{L^{\infty}}+1)}\|u\|_{L^{\infty}(\hball{\frac{r}{2}}{\bar{x}_{1}})}^{\mu^{\bar{m}}}\\
    &\leq e^{C_{r}(\|H\|^{\frac{2}{3}}_{L^{\infty}} + \|h\|^{2}_{L^{\infty}}+1)}\|u\|^{\mu^{\bar{m}}}_{L^{\infty}(\hball{r}{{x_0}})}.
    \end{aligned} 
    \end{equation*}
 The combination of (\ref{kao-j}) and the last inequality give the desired inequality in the lemma.
\end{proof}

Given the Carleman estimates and the growth lemma (i.e. Lemma \ref{lemma Gro}), we are able to prove the doubling inequality in half balls.

\begin{proposition}\label{prop-2}
 Let $u$ be a solution to \eqref{bound PDE}. There exist positive constants $C$ and $r_{0}$ depending only on the manifold $\mc{M}$ such that
\begin{equation}\label{half doubling}
\| u \|_{L^{\infty}(\mbb{B}^{+}_{2r}(x_{0}))} \leq e^{C(\|H\|^{\frac{2}{3}}_{L^{\infty}} + \|h\|^{2}_{L^{\infty}})}\| u \|_{L^{\infty}(\mbb{B}^{+}_{r}(x_{0}))}
\end{equation}
for any $0<2r<r_{0}$ and any $\mbb{B}_{2r}^{+}(x_0) \subset \mc{M}$ and $x_{0} \in \partial\mc{M}$.   
\end{proposition}

\begin{proof}
    Without loss of generality, let $x_{0} = 0$ and $r_{0} = 1$.  Fix $r$ such that $0<  2r < r_{0} = 1$. Take $0 < \rho < \frac{r}{12}$.  
Let $u$ be a solution of \eqref{bound PDE}. We introduce a radial smooth cut-off function $\eta(|x|)\in C^{\infty}_{0}(\mbb{B}^{+}_{1})$ such that $0 < \eta < 1$ satisfies $\eta(|x|) = 1$ in $A_{2\rho,r}$, $\eta(|x|) = 0$ outside of $ A_{\rho,\frac{3r}{2}} $, and $|\nabla^{\alpha} \eta(x) | \leq C$ for some constant $C$. Then $\eta u$ is the weak solution of (\ref{new-pde}), Applying the Carleman estimate \eqref{car est} to $\eta u$, we have
\begin{equation*}
\begin{aligned}
       \|r^{2}e^{\tau\psi}(2\nabla \eta \cdot \nabla u + \triangle\eta u)\|_{\mathbb{B}_{1}^{+}}^{2} \geq C\tau^2 \rho\|e^{\tau\psi} r^{-1/2}\eta u\|^{2}_{\mathbb{B}_{1}^{+}}+ C\tau^3 \|e^{\tau\psi} (\ln r)^{-1} \eta u\|^{2}_{\mathbb{B}_{1}^{+}}.
    \end{aligned} 
\end{equation*}
 Using the properties of $\eta$, we  get
\begin{equation*}
\begin{aligned}
       \|re^{\tau\psi}\nabla u\|_{A_{\rho,2\rho}} + \|re^{\tau\psi}\nabla u\|_{A_{r,\frac{3r}{2}}} &+ \|e^{\tau\psi}u\|_{A_{\rho,2\rho}} + \|e^{\tau\psi}u\|_{A_{r,\frac{3r}{2}}} \\&\geq C\tau \rho^{1/2}\|e^{\tau\psi}  r^{-1/2}u\|_{A_{2\rho,6\rho}} + C\tau^{\frac{3}{2}} \|e^{\tau\psi}  (\ln r)^{-1}u\|_{A_{\frac{r}{2}},r}\\
       &\geq  C\tau \|e^{\tau\psi}  u\|_{A_{2\rho,6\rho}} + C\tau^{\frac{3}{2}}\|e^{\tau\psi}  u\|_{A_{\frac{r}{2},r}},
    \end{aligned} 
\end{equation*}
where we used the fact that $r$ is fixed.
We then take a maximum and minimum of $\psi$ and $r$ over each annuli to get
\begin{equation*}
\begin{aligned}
       e^{\tau\psi_{1}}(\|\nabla u\|_{A_{\rho,2\rho}} &+ \|u\|_{A_{\rho,2\rho}}) + e^{\tau\psi_{2}}(\|\nabla u\|_{A_{r,\frac{3r}{2}}} + \|u\|_{A_{r,\frac{3r}{2}}}) \\ &\geq Ce^{\tau\psi_{3}}\|  u\|_{A_{2\rho,6\rho}} + Ce^{\tau\psi_{4}}\|  u\|_{A_{\frac{r}{2},\frac{4r}{5}}}, \\
    \end{aligned} 
\end{equation*}
where $\psi_{1} = \max_{A_{\rho,2\rho}} \psi$, $\psi_{2} = \max_{A_{r,\frac{3r}{2}}} \psi$, $\psi_{3} = \min_{A_{2\rho,6\rho}} \psi$, and $\psi_{4} = \min_{A_{\frac{r}{2}, \frac{4r}{5}}} \psi$. Note that $\psi_{1}-\psi_{3}>0$ and $\psi_{4}-\psi_{2}>0$, which can be chosen to be independent of $r$ or $\rho$.  Applying the Caccioppoli inequality (\ref{corollary Cac})  and enlarging each annulus to a ball gives
\begin{equation*}
\begin{aligned}
       C(1+\|H\|_{L^{\infty}}^{\frac{1}{2}}+\|h\|_{L^{\infty}})(e^{\tau\psi_{1}}\|u\|_{\chball{3\rho}} + e^{\tau\psi_{2}}\|u\|_{\chball{2r}}) \geq e^{\tau\psi_{3}}\|  u\|_{\chball{6\rho}} + e^{\tau\psi_{4}}\|  u\|_{A_{\frac{r}{2},\frac{4r}{5}}},\\
    \end{aligned} 
\end{equation*}
where we have enlarged the right hand side by adding $e^{\tau\psi_{3}}\|u\|_{\chball{2\rho}}$ to both sides. We now want to incorporate the $\|u\|_{\chball{2r}}$ term into the right hand side. To this end, we choose $\tau$ such that
$$
C(1+\|H\|_{L^{\infty}}^{\frac{1}{2}}+\|h\|_{L^{\infty}})e^{\tau\psi_{2}}\|u\|_{\chball{2r}} \leq \frac{1}{2}e^{\tau\psi_{4}}\|  u\|_{A_{\frac{r}{2},\frac{4r}{5}}},  
$$
which holds if
$$
\tau \geq \frac{1}{\psi_{4} - \psi_{2}} \ln{\frac{2C (1+\|H\|_{L^{\infty}}^{\frac{1}{2}}+\|h\|_{L^{\infty}})\|u\|_{\chball{2r}}}{\|  u\|_{A_{\frac{r}{2},\frac{4r}{5}}} }}.
$$
It follows that
\begin{equation}\label{hbd 1}
\begin{aligned}
       C(1+\|H\|_{L^{\infty}}^{\frac{1}{2}}+\|h\|_{L^{\infty}})e^{\tau\psi_{1}}\|u\|_{\chball{3\rho}} \geq e^{\tau\psi_{3}}\|  u\|_{\chball{6\rho}} + e^{\tau\psi_{4}}\|  u\|_{A_{\frac{r}{2},\frac{4r}{5}}}. \\
    \end{aligned} 
\end{equation}
Recall that $\tau > C(1+\|H\|^{\frac{2}{3}}_{L^{\infty}} + \|h\|_{L^{\infty}}^{2})$ is required to apply the Carleman estimate \eqref{car est}.  Thus, we choose
$$
\tau = C(1+\|H\|^{\frac{2}{3}}_{L^{\infty}} + \|h\|_{L^{\infty}}^{2}) + \frac{1}{\psi_{4} - \psi_{2}} \ln{\frac{2C (1+\|H\|_{L^{\infty}}^{\frac{1}{2}}+\|h\|_{L^{\infty}})\|u\|_{\chball{2r}}}{\|  u\|_{A_{\frac{r}{2},\frac{4r}{5}}} }}.$$
Substituting this value of $\tau$ into \eqref{hbd 1} and dropping the second term on the right hand side gives
\begin{equation*}
\begin{aligned}
       e^{C(1+\|H\|^{\frac{2}{3}}_{L^{\infty}} + \|h\|_{L^{\infty}}^{2})}(\frac{\|u\|_{\chball{2r}}}{\|  u\|_{A_{\frac{r}{2},\frac{4r}{5}}}})^{C}\|u\|_{\chball{3\rho}} \geq \|  u\|_{\chball{6\rho}}. \\
    \end{aligned} 
\end{equation*}
Notice that there exists a point $x_{1} \in \partial\mc{M}$ such that $\hball{\frac{r}{4}}{x_{1}} \subset A_{\frac{r}{2},\frac{4r}{5}}$. We apply the growth Lemma \ref{lemma Gro} to see that 
$$
\begin{aligned}
\frac{\|u\|_{\chball{2r}}}{\|  u\|_{A_{\frac{r}{2},\frac{4r}{5}}}} \leq \frac{\|u\|_{\chball{2r}}}{\|  u\|_{\hball{\frac{r}{4}}{x_{1}}}} &\leq e^{C_{r}(1+\|H\|_{L^{\infty}}^{\frac{2}{3}}+\|h\|_{L^{\infty}}^{2})} \frac{\|u\|_{\chball{2r}}}{\|  u\|_{\mc{M}}}\\
& \leq e^{C_{r}(1+\|H\|_{L^{\infty}}^{\frac{2}{3}}+\|h\|_{L^{\infty}}^{2})}.
\end{aligned}
$$
Thus, for $\rho < \frac{r}{12}$ and fixed value of $r$, we obtain that
\begin{equation}
\begin{aligned}
       \|  u\|_{\chball{6\rho}} \leq e^{C(1+\|H\|^{\frac{2}{3}}_{L^{\infty}} + \|h\|_{L^{\infty}}^{2})}\|u\|_{\chball{3\rho}}.\label{dou-1}
    \end{aligned} 
\end{equation}
If $\rho \geq \frac{r}{12}$, since $r$ is fixed, we can get directly from the growth lemma
that
\begin{equation}
\begin{aligned}
       e^{-C_{r}(1+\|H\|^{\frac{2}{3}}_{L^{\infty}} + \|h\|_{L^{\infty}}^{2})}\|  u\|_{\chball{6\rho}} &\leq e^{-C_{r}(1+\|H\|^{\frac{2}{3}}_{L^{\infty}} + \|h\|_{L^{\infty}}^{2})}\|  u\|_{\mc{M}}  \\
       &\leq \|u\|_{\chball{\frac{r}{12}}} \leq \|u\|_{\chball{3\rho}}.\label{dou-2}
    \end{aligned} 
\end{equation}
The combination of \eqref{dou-1} and \eqref{dou-2} yields the desired doubling inequality in the Proposition.
\end{proof}

\section{Boundary Doubling inequality}
In order to show the boundary doubling inequality, we need the following global Carleman estimate involving the boundary.  We choose a weight function
\begin{align}
   \psi (x) = e^{sm(x)} 
   \label{weight-1}
\end{align}
where $m(x) = -|x - b|,$ and $b = (0,...,0,-b_{n})$ with  some small positive $b_{n} > 0$.

\begin{proposition}\label{prop GCE}
 Let $s$ be a fixed large constant and $v\in H^1(\mbb{B}^{+}_{2})$ is the solution of 
 \begin{equation}
    \begin{cases}
        \Delta_{g} v = F&  \hspace{0.5cm} in \hspace{0.1cm} \mathbb{B}_{2}^{+}, \\
         \frac{\partial v} {\partial {\nu}} = g& \hspace{0.5cm} on \hspace{0.1cm} \partial\mathbb{B}_{2}^{+}\cap \{x  |  x_{n} = 0\} 
         \label{des-1}
    \end{cases}
\end{equation}
 with $\supp v\subset \mbb{B}^{+}_{2}$, $F\in L^2(\mathbb{B}_{2}^{+})$, and $g\in L^2(\partial\mathbb{B}_{2}^{+}\cap \{x  |  x_{n} = 0\}).$   Then there exists a positive constant $C_{0}$ depending on $s$ such that 
\begin{equation}\label{carleman}
\begin{aligned}
     &\|e^{\tau\psi} F\|_{L^{2}(\mbb{B}^{+}_{2})}+ \tau^{\frac{3}{2}}s^{2}\|\psi^{\frac{3}{2}}e^{\tau\psi}v\|_{L^{2}(\partial\mbb{B}^{+}_{2}\cap \{x  |  x_{n} = 0\})} + \tau^{\frac{1}{2}}\|e^{\tau\psi}g\|_{L^{2}(\partial\mbb{B}^{+}_{2}\cap \{x  |  x_{n} = 0\})} \\
   & + \tau^{\frac{1}{2}}s\|\psi^{\frac{1}{2}}e^{\tau\psi}\nabla' v\|_{L^{2}(\partial\mbb{B}^{+}_{2}\cap \{x  |  x_{n} = 0\})}
    \\
    &\geq C_{0}\tau^{\frac{3}{2}}s^{2}\|\psi^{\frac{3}{2}}e^{\tau\psi}v\|_{L^{2}(\mbb{B}^{+}_{2})} + C_{0}\tau^{\frac{1}{2}}s\|\psi^{\frac{1}{2}}e^{\tau\psi}\nabla v\|_{L^{2}(\mbb{B}^{+}_{2})}.
    \end{aligned}
\end{equation}
\end{proposition}
Note that $\nabla'$ is the derivative for $x'=(x_1, x_2, \cdots, x_{n-1})$ on  $\partial\mathbb{B}_{2}^{+}\cap \{x  |  x_{n} = 0\}$. The proof of the Carleman estimates (\ref{carleman}) will be given in section 4. Thanks to this Carleman estimate, we can show a Cauchy uniqueness result (i.e. two half-ball and one lower dimensional ball inequality inequality).

\begin{lemma}\label{3lower}
Let $u$ be a solution of \eqref{bound PDE} in $\mbb{B}^{+}_{1/2}$. Denote the lower dimensional ball on the boundary
$$
\mathbf{B}_{1/3} = \{(x',0)\in\mbb{R}^{n}|x'\in\mbb{R}^{n-1}, |x'| < 1/3\}.
$$
Then there exists a constant $0<\kappa<1$ such that
\begin{equation}\label{three ball}
    \| u\|_{L^{2}(\mbb{B}^{+}_{1/12})} \leq e^{C(1+\|H\|^{\frac{2}{3}}_{L^{\infty}} +\|h\|^{{2}}_{L^{\infty}})}\|u\|_{L^{2}(\mbb{B}^{+}_{1/2})}^{\kappa} (\| u\|_{H^{1}(\mathbf{B_{1/3}})} )^{1-\kappa}.
\end{equation}

\end{lemma}

\begin{proof}[Proof of Lemma \ref{3lower}]
Take a smooth radial cutoff function $\eta$ such that $\eta(|x|) = 1$ in $\mbb{B}^{+}_{1/8}$, $\eta(|x|) = 0$ outside of $\mbb{B}^{+}_{1/4}$, and $|\nabla \eta(|x|) | \leq C$ for some constant $C$.  Let $u$ is the solution of  \eqref{bound PDE} in $\mathbb{B}_{2}^+$, then  $w=\eta u$ is the weak solution of 
\begin{equation}\label{new-pde-1}
     \begin{cases} 
         \Delta w = 2\nabla u\cdot\nabla \eta+\triangle\eta u+H(x)w&  \hspace{0.5cm} in \hspace{0.1cm} \mathbb{B}_{2}^{+}, \\
        \frac{\partial w}{ \partial{\nu} } = h(x) w& \hspace{0.5cm} on \hspace{0.1cm} \partial\mathbb{B}_{2}^{+}\cap \{x \hspace{0.1cm} | \hspace{0.1cm} x_{n} = 0\}.
   \end{cases}
\end{equation}
Notice that $s$ is fixed and $\psi$ is bounded above and below.  Applying the Carleman estimates (\ref{carleman}) 
with $F=2\nabla u\cdot\nabla \eta+\triangle\eta u+H(x)w$ and $g= h(x) w $, we get
\begin{equation}\label{eta carl}
\begin{aligned}
    \|e^{\tau\psi}(2\nabla u\cdot\nabla \eta&+\triangle\eta u+H(x)w )\|_{L^{2}(\mbb{B}^{+}_{1/4})} + \tau^{\frac{1}{2}}\|e^{\tau\psi} hw\|_{L^{2}(\mathbf{B_{1/4}})}\\
   & + \tau^{\frac{1}{2}}\|e^{\tau\psi} \nabla' w\|_{L^{2}(\mathbf{B_{1/4}})}+ \tau^{\frac{3}{2}}\|e^{\tau\psi} w\|_{L^{2}(\mathbf{B_{1/4}})}\\
    &\geq C\tau^{\frac{3}{2}}\|e^{\tau\psi} w\|_{L^{2}(\mbb{B}^{+}_{1/4})}.
    \end{aligned}
\end{equation}
Choosing 
\begin{align} \label{tau-1}
\tau > C(1 + \|H\|^{\frac{2}{3}}_{L^{\infty}} +\|h\|_{L^{\infty}}),
\end{align}
we obtain that
\begin{equation*}
\begin{aligned}
    \|e^{\tau\psi}(2\nabla u\cdot\nabla \eta&+\triangle\eta u)\|_{L^{2}(\mbb{B}^{+}_{1/4})} + 
     \tau^{\frac{1}{2}}\|e^{\tau\psi} \nabla' u\|_{L^{2}(\mathbf{B_{1/4}})}+ \tau^{\frac{3}{2}}\|e^{\tau\psi} u\|_{L^{2}(\mathbf{B_{1/4}})}\\
    &\geq C\tau^{\frac{3}{2}}\|e^{\tau\psi} u\|_{L^{2}(\mbb{B}^{+}_{1/8})}.
    \end{aligned}
\end{equation*}
From the properties of $\eta$, we get
\begin{equation}\label{bounding 5}
\begin{aligned}
    \|e^{\tau\psi}u\|_{L^{2}(\mbb{B}^{+}_{1/4} \setminus \mbb{B}^{+}_{1/8})}& + \|e^{\tau\psi}\nabla u\|_{L^{2}(\mbb{B}^{+}_{1/4}\setminus \mbb{B}^{+}_{1/8})}   + \tau^{\frac{3}{2}}\|e^{\tau\psi} u\|_{L^{2}(\mathbf{B_{1/4}})}
     \\&+ \tau^{\frac{1}{2}}\|e^{\tau\psi} \nabla' u\|_{L^{2}(\mathbf{B_{1/4}})}\\
   & \geq C\tau^{\frac{3}{2}}\|e^{\tau\psi} u\|_{L^{2}(\mbb{B}^{+}_{1/8})}.
    \end{aligned}
\end{equation}
We aim to find the maximum and minimum of $\psi$ on left hand side and right hand side of \eqref{bounding 5}. By definition of $\psi$ and choosing $b_n$ appropriately small, we see that $\max_{\mbb{B}^{+}_{1/4} \setminus \mbb{B}^{+}_{1/8} }(m(x)) = -\sqrt{ (\frac 1 {8})^2+b^2_{n}}\leq -(\frac{1}{10}+b_n))$  and $\min_{\mbb{B}^{+}_{1/12}}(m(x)) = -(\frac{1}{12} + b_{n})$.
Plugging in these bounds, we arrive at
\begin{equation}\label{bounding 6}
\begin{aligned}
    e^{\tau e^{-s(\frac 1 {10}+b_{n})}}\|u\|_{L^{2}(\mbb{B}^{+}_{1/4})} + e^{\tau e^{-s(\frac 1 {10}+b_{n})}}\|\nabla u\|_{L^{2}(\mbb{B}^{+}_{1/4})} &+ \tau^{\frac{3}{2}}e^{\tau}\| u\|_{L^{2}(\mathbf{B_{1/4}})} \\
    &+ \tau^{\frac{1}{2}}e^{\tau}\| \nabla' u\|_{L^{2}(\mathbf{B_{1/4}})}\\
    &\geq C\tau^{\frac{3}{2}}e^{\tau e^{-s(\frac{1}{12} + b_{n})}}\| u\|_{L^{2}(\mbb{B}^{+}_{1/12})}. 
    \end{aligned}
\end{equation}
The  Caccioppoli inequality from Lemma \ref{corollary Cac} yields that
\begin{equation}\label{bounding 7}
\begin{aligned}
    e^{\tau e^{-s(\frac 1 {10}+b_{n})}}\|u\|_{L^{2}(\mbb{B}^{+}_{1/2})} + e^{\tau}\| u\|_{L^{2}(\mathbf{B_{1/3}})}
    + e^{\tau}\| \nabla' u\|_{L^{2}(\mathbf{B_{1/3}})}\\
    \geq Ce^{\tau e^{-s(\frac{1}{12} + b_{n})}}\| u\|_{L^{2}(\mbb{B}^{+}_{1/12})}. 
    \end{aligned}
\end{equation}
Let
$$
B_{1} = \|u\|_{L^{2}(\mbb{B}^{+}_{1/2})},
$$
$$
B_{2} = \| u\|_{L^{2}(\mathbf{B_{1/3}})}
    +\|\nabla' u\|_{L^{2}(\mathbf{B_{1/3})}},
$$
$$
B_{3} = \| u\|_{L^{2}(\mbb{B}^{+}_{1/12})}.
$$
Set $p_{0} =  e^{ e^{-s(\frac 1 {10}+b_{n})}} - e^{-s(\frac{1}{12}+b_{n})}$ and $p_{1} = 1 - e^{-s(\frac{1}{12} + b_{n})}$. Note that $p_0<0$ and $p_1>0$. We  get
\begin{equation}\label{B ineq}
    e^{\tau p_{0}}B_{1} + e^{\tau p_{1}}B_{2} \geq CB_{3}.
\end{equation}
To incorporate the $B_{1}$ term into the right hand side, we take
$$
e^{\tau p_{0}}B_{1} \leq \frac{1}{2}CB_{3}.
$$
Then we need
\begin{align}\label{english}
\tau \geq \frac{1}{p_0}\ln (\frac{CB_{3}}{2B_{1}}).
\end{align}
By \eqref{B ineq}, for such $\tau$, we have that 
\begin{equation}\label{B2 B3 ineq}
    e^{\tau p_1}B_{2} \geq \frac{1}{2}CB_{3}.
\end{equation}
Recalling the assumptions for $\tau$ in (\ref{tau-1}) and in (\ref{english}), we take
\begin{equation*}
    \tau = C(1+\|H\|^{\frac{2}{3}}_{L^{\infty}} +\|h\|_{L^{\infty}}) + \frac{1}{p_0}\ln (\frac{CB_{3}}{2B_{1}}).
\end{equation*}
Substituting this choice of $\tau$ into \eqref{B2 B3 ineq} to get
\begin{equation*}
    e^{C(1+\|H\|^{\frac{2}{3}}_{L^{\infty}} +\|h\|_{L^{\infty}})}B_{1}^{\frac{p_1}{p_1 - p_0}} B_{2}^{\frac{-p_0}{p_1 - p_0}} \geq CB_{3}.
\end{equation*}
Let $\kappa = \frac{p_1}{p_1 - p_0}$.  Then,  the three-ball type inequality follows 
\begin{equation}\label{3ball-ish}
    \| u\|_{L^{2}(\mbb{B}^{+}_{1/12})} \leq e^{C(1+\|H\|^{\frac{2}{3}}_{L^{\infty}} +\|h\|_{L^{\infty}})}\|u\|_{L^{2}(\mbb{B}^{+}_{1/2})}^{\kappa} (\| u\|_{L^{2}(\mathbf{B_{1/3}})}
    +\|\nabla' u\|_{L^{2}(\mathbf{B_{1/3})}})^{1-\kappa}.
\end{equation}
  Therefore, the lemma is finished.
\end{proof}

Thanks to the quantitative Cauchy uniqueness, we are able to present the proof of Theorem \ref{thm 1}. However, we will need an interpolation estimate (\ref{interp ineq}) which requires $H^2$ estimate near the boundary.  Because of the lack of $H^2$ regularity for $u$ in (\ref{bound PDE}), we employ a lifting argument to consider a new elliptic equation with zero Neumann boundary condition. The $H^2$ estimate can be achieved for the solutions of the new elliptic equation.


\begin{proof}[Proof of Theorem \ref{thm 1}]

We consider the following auxiliary function 
\begin{equation*}
     \begin{cases} 
      \triangle w(x)=0 \quad & x\in  \mathbb B^+_4, \\
    \frac{\partial  w }{\partial \nu}= \hat{h}(x), & x\in \partial\mathbb B^+_4,
  \end{cases}
\end{equation*}
where $\hat{h}$ is given as
\begin{equation*}
   \hat{h}(x)=  \left\{ 
   \begin{array}{lll}
     h(x) \quad \mbox{in}\   \bar{\mathbb B}^+_4\cap \{x|x_n=0\}, \\
   \frac{-1}{|\partial \bar{\mathbb B}^+_4\backslash \{x|x_n=0\}|  }\int_{\partial \bar{\mathbb B}^+_4\cap \{x|x_n=0\}}  h(x) dS \quad \mbox{elsewhere on} \ \partial\mathbb B^+_4.
   \end{array}
  \right.
\end{equation*}
The existence of $w$ is known, see e.g. \cite{JK81}.
Since   $\|\hat{h}(x)\|_{L^\infty (\partial \mathbb B^+_4) }\leq \|{h}(x)\|_{L^\infty (\partial\mathbb B^+_4\cap \{x|x_n=0\}) }$, by elliptic estimates, it holds that
\begin{align}
    \|w\|_{W^{1,p}(\mathbb B^+_4) }\leq C\|{h}(x)\|_{L^\infty (\partial \mathbb B^+_4\cap\{x|x_n=0\}) }
   \label{regu-w}
\end{align}
for any $1<p<\infty$. Furthermore, Lemma \ref{lem-re} shows that
\begin{align}
    \|w\|_{H^{\frac{3}{2}}(\mathbb B^+_4) }\leq C \|\hat{h}(x)\|_{L^2 (\partial \mathbb B^+_4) }\leq C\|{h}(x)\|_{L^\infty (\partial \mathbb B^+_4 \cap\{x|x_n=0\}) }.
   \label{regu-w-1}
\end{align}

From (\ref{regu-w}), the Sobolev imbedding theorem gives that
\begin{align}
    \|w\|_{L^\infty(\mathbb B^+_4) }\leq C\|{h}(x)\|_{L^\infty (\partial\mathbb B^+_4\cap\{x|x_n=0\} ) }
    \label{infty-1}.
\end{align}

We construct an auxiliary function
\begin{align} 
\bar u= e^{-w(x)}u.
\label{rela-u}\end{align} We can check that
\begin{align*}
    \frac{\partial \bar u}{\partial \nu}=0 \quad \mbox{on} \ \bar{\mathbb B}^+_2\cap \{x|x_n=0\}.
\end{align*}
We can also show that $\bar u$ is the weak solution of 
\begin{align*}
    \triangle \bar u+2\nabla w\cdot \nabla \bar u +|\nabla w|^2 \bar u=H(x)\bar u \quad \mbox{in} \ {\mathbb B}^+_2.
\end{align*}
Thus, $\bar u$ satisfies the equation 
\begin{equation} 
    \begin{cases} 
      \triangle \bar u+2\nabla w\cdot \nabla \bar u +|\nabla w|^2 \bar u=H(x)\bar u \quad \mbox{in} \ {\mathbb B}^+_2, \\
    \frac{\partial \bar u}{\partial \nu}=0 \quad \mbox{on} \ \bar{\mathbb B}^+_2\cap \{x|x_n=0\}.
   \end{cases}
  \label{bar-u}
\end{equation}
By elliptic estimates, the solution $u$ in (\ref{bound PDE}) in ${\mathbb B}^+_2$ satisfies 
\begin{align}
    \|u\|_{W^{1, p}(\mathbb B^+_1)}\leq C(\|H\|_\infty,  \|h\|_\infty, p)\|u\|_{L^2(\mathbb B^+_2)}
   \label{p-regu}
\end{align}
 for any $1<p<\infty$, where $C(\|H\|_\infty,  \|h\|_\infty, p)$ depends on $\|H\|_\infty,  \|h\|_\infty$ and $p$.
From (\ref{p-regu}), (\ref{regu-w}) and the definition of $\bar u$, we get that $2\nabla w\cdot \nabla \bar u +|\nabla w|^2 \bar u-H(x)\bar u \in L^2(\mathbb B^+_2).$ Therefore, $\bar u \in H^2(\mathbb B^+_2)$. Next we show the $H^2$ norm of $\bar u$ is bounded.

By elliptic estimates, (\ref{rela-u}) and rescaling argument for (\ref{p-regu}), we can show that
\begin{align}
   \|\bar u\|_{L^\infty(\mathbb B^+_1)}&\leq e^{\|h\|_\infty} \| u\|_{L^\infty(\mathbb B^+_1)} \nonumber \\
  & \leq C (\|H\|_\infty^{\frac{n}{2}}+\|h\|_\infty^{n}+1) e^{\|h\|_\infty} \| u\|_{L^2(\mathbb B^+_2)} \nonumber \\
 &\leq  C (\|H\|_\infty^{\frac{n}{2}}+\|h\|_\infty^{n}+1) e^{C\|h\|_\infty} \| \bar u\|_{L^\infty(\mathbb B^+_2)}.
 \label{324}
\end{align}
Hence, it follows from (\ref{regu-w}) that
\begin{align*}
     \||\nabla w|^2\bar u\|_{L^2(\mathbb B^+_1)}\leq \| \bar u\|_{L^\infty(\mathbb B^+_2)}\|\nabla w\|^2_{L^4(\mathbb B^+_2)} \leq \| \bar u\|_{L^\infty(\mathbb B^+_2)} \|h\|^2_\infty.
\end{align*}
From the assumption of $H$, it holds that
\begin{align*}
 \|H(x)\bar u \|_{L^2(\mathbb B^+_1)}\leq  \|H\|_\infty \|\bar u \|_{L^2(\mathbb B^+_2)}.
\end{align*}

By the rescaling arguments for (\ref{p-regu}), covering arguments,  and elliptic estimates,
we can obtain that
\begin{align*}
    \|\nabla u\|_{L^4(\mathbb B^+_1)}\leq C (\|H\|_\infty^{\frac{1}{2}}+\|h\|_\infty+1)^{2n+1}   \| u\|_{L^2(\mathbb B^+_2)}.
\end{align*}
Thanks to (\ref{regu-w}), the relation of $\bar u$ and $u$ in (\ref{rela-u}), and the last inequality, we derive that
\begin{align*}
\|\nabla \bar u\cdot \nabla w\|_{L^2(\mathbb B^+_1) }&\leq \|\nabla \bar u\|_{L^4(\mathbb B^+_1) } \|\nabla w\|_{L^4(\mathbb B^+_1) } \nonumber \\
&\leq C\|\nabla \bar u\|_{L^4(\mathbb B^+_1) } \|h\|_\infty \nonumber \\
&\leq C\|\nabla u e^{-w}- u e^{-w} \nabla w\|_{L^4(\mathbb B^+_1) } \|h\|_\infty \nonumber \\
&\leq e^{C(\|h\|_\infty+1)} (\|\nabla u \|_{L^4(\mathbb B^+_1)}+\|u \nabla w\|_{L^4(\mathbb B^+_1)} ) \nonumber \\
&\leq e^{C(\|h\|_\infty+1)} (\|\nabla u \|_{L^4(\mathbb B^+_1)}+\|u\|_{L^\infty(\mathbb B^+_1)} \|\nabla w\|_{L^4(\mathbb B^+_1)} ) \nonumber \\
&\leq e^{C(\|h\|_\infty+1)} \big(\|H\|_\infty^{\frac{1}{2}}+\|h\|_\infty+1)^{2n+1}   \| u\|_{L^2(\mathbb B^+_2)}   +\|u\|_{L^2(\mathbb B^+_2) }\|h\|_\infty \big)\nonumber \\
&\leq   e^{C(\|H\|^{\frac{1}{2}}_\infty +\|h\|_\infty+1)}  \| u\|_{L^2(\mathbb B^+_2)} \nonumber \\
&\leq  e^{C (\|H\|^{\frac{1}{2}}_\infty +\|h\|_\infty+1)}  \| \bar u\|_{L^2(\mathbb B^+_2)}.
\end{align*}
Therefore, by the elliptic estimates for (\ref{bar-u}), we  obtain that
\begin{align}
    \|\bar u\|_{H^2(\mathbb B^+_1) }\leq C e^{(\|H\|^{\frac{1}{2}}_\infty +\|h\|_\infty+1)}  \| \bar u\|_{L^2(\mathbb B^+_2)}.
    \label{strong-1}
\end{align}

Next we will derive a boundary doubling estimate for $\bar u$.
We consider $\bar{u}$ be the solution to \eqref{bar-u}.  We may normalize $\bar{u}$ by rescaling as
\begin{equation}
   \|\bar{u}\|_{L^{2}(\mbb{B}_{1/2}^{+})} = 1.
\end{equation}
We claim there exists a positive constant $C > 0$ such that
\begin{equation}\label{claim 1}
  \|\bar{u}\|_{L^\infty(\mathbf B_{1/6})}+  \|\nabla' \bar{u}\|_{L^2(\mathbf B_{1/6})} \geq e^{-C(1+\|H\|^{\frac{2}{3}}_{L^{\infty}} + \|h\|^{2}_{L^{\infty}})}.
\end{equation}
To see this, we apply the Cauchy uniqueness result in \eqref{three ball}. The normalized inequality becomes
\begin{equation}
    \| {u}\|_{L^{2}(\mbb{B}^{+}_{1/16})} \leq e^{C(1+\|H\|^{\frac{2}{3}}_{L^{\infty}} + \|h\|^{2}_{L^{\infty}})}\|{u}\|_{L^{2}(\mbb{B}^{+}_{1/4})}^{\kappa} \|{u}\|_{H^{1}(\mathbf{B_{1/6}})}^{1-\kappa}.
    \label{three-1}
\end{equation}
It holds that $\nabla' u=e^{w(x)}\nabla'w \bar u+ e^{w(x)}\nabla' \bar u$. It follows from (\ref{regu-w-1}) and trace inequality (\ref{trace}) that
\begin{align}
   \|\nabla' {u}\|_{L^2(\mathbf{B_{1/6}})}&\leq e^{C(1+\|H\|^{\frac{2}{3}}_{L^{\infty}} + \|h\|^{2}_{L^{\infty}})} (\|\bar u\|_{L^\infty(\mathbf{B_{1/6}}) }\|\nabla' w\|_{L^2(\mathbf{B_{1/6}})}+ \|\nabla' \bar u\|_{L^2(\mathbf{B_{1/6}})}) \nonumber\\
   &\leq e^{C(1+\|H\|^{\frac{2}{3}}_{L^{\infty}} + \|h\|^{2}_{L^{\infty}})} (\|\bar u\|_{L^\infty(\mathbf{B_{1/6}})} \| w\|_{H^{\frac{3}{2}}(\mathbf{B_{1/6}})}+  \|\nabla' \bar u\|_{L^2(\mathbf{B_{1/6}})}) \nonumber\\
   &\leq e^{C(1+\|H\|^{\frac{2}{3}}_{L^{\infty}} + \|h\|^{2}_{L^{\infty}})} (\|\bar u\|_{L^\infty(\mathbf{B_{1/6}})}+  \|\nabla' \bar u\|_{L^2(\mathbf{B_{1/6}})}).
   \label{gra-1}
\end{align}
From the relation of $\bar u$ and $u$ in (\ref{rela-u}) and (\ref{half doubling}), we also get
\begin{align*}
    e^{-C(1+  \|h\|_{L^{\infty}})}  \|  {u}\|_{L^{2}(\mbb{B}^{+}_{R})}  \leq \| \bar {u}\|_{L^{2}(\mbb{B}^{+}_{R})}\leq e^{C(1+  \|h\|_{L^{\infty}})}  \|  {u}\|_{L^{2}(\mbb{B}^{+}_{R})}
\end{align*}
and 
\begin{equation}\label{half-d}
\| \bar u \|_{L^{\infty}(\mbb{B}^{+}_{2R}(x_{0}))} \leq e^{C(\|H\|^{\frac{2}{3}}_{L^{\infty}} + \|h\|^{2}_{L^{\infty}})}\|\bar u \|_{L^{\infty}(\mbb{B}^{+}_{R}(x_{0}))}
\end{equation}
for any $0<R<4$ under rescalings. Thus, (\ref{three-1}) and (\ref{gra-1}) imply that
\begin{equation*}
    \|\bar {u}\|_{L^{2}(\mbb{B}^{+}_{1/16})} \leq e^{C(1+\|H\|^{\frac{2}{3}}_{L^{\infty}} + \|h\|^{2}_{L^{\infty}})}\|\bar {u}\|_{L^{2}(\mbb{B}^{+}_{1/4})}^{\kappa} (\|\bar u\|_{L^\infty(\mathbf{B_{1/6}})}+  \|\nabla' \bar u\|_{L^2(\mathbf{B_{1/6}})})^{1-\kappa}.
\end{equation*}
Applying the doubling inequality (\ref{half-d}) in half balls finite many times, we have
\begin{align*}
    \|\bar {u}\|_{L^{2}(\mbb{B}^{+}_{1/4})}\leq e^{C(1+\|H\|^{\frac{2}{3}}_{L^{\infty}} + \|h\|^{2}_{L^{\infty}})} \|\bar {u}\|_{L^{2}(\mbb{B}^{+}_{1/16})}.
\end{align*}
Thus, the last two inequalities give that
\begin{align}
     \|\bar {u}\|_{L^{2}(\mbb{B}^{+}_{1/16})} \leq e^{C(1+\|H\|^{\frac{2}{3}}_{L^{\infty}} + \|h\|^{2}_{L^{\infty}})} (\|\bar u\|_{L^\infty(\mathbf{B_{1/6}})}+  \|\nabla' \bar u\|_{L^2(\mathbf{B_{1/6}})}).\label{yield-1}
\end{align}
We can apply the doubling inequality  (\ref{half-d}) finitely many times to have
\begin{equation}\label{yield-2}
    \begin{aligned}
        \|\bar{u}\|_{L^{2}(\mbb{B}^{+}_{1/16})} &\geq e^{-C(1+\|H\|^{\frac{2}{3}}_{L^{\infty}} +\|h\|^{2}_{L^{\infty}})}\|\bar{u}\|_{L^{2}(\mbb{B}^{+}_{1/2})}\\
        & \geq e^{-C(1+\|H\|^{\frac{2}{3}}_{L^{\infty}} +\|h\|^{2}_{L^{\infty}})}.
    \end{aligned}
\end{equation}
 Thus, the combination of (\ref{yield-1}) and (\ref{yield-2}) show that the claim (\ref{claim 1}) is true.
 
We now claim there exists $C>0$ such that
\begin{equation}\label{claim 2}
    \|\bar{u}\|_{L^{\infty}(\mathbf B_{1/5})} \geq e^{-C(1+\|H\|^{\frac{2}{3}}_{L^{\infty}} + \|h\|^{2}_{L^{\infty}})}.
\end{equation}
Recall the following interpolation estimate from \cite{R17}. For any $0<\varepsilon<1$, there holds
\begin{equation}\label{interp ineq}
    \|\nabla' w_1\|_{L^{2}(\mbb{R}^{n-1})} \leq \varepsilon \|w_1\|_{H^2(\mbb{R}^{n}_{+})}  + C\varepsilon^{-2} \|w_1\|_{L^{2}(\mbb{R}^{n-1})}.
\end{equation}
Note that $H^2$ regularity in (\ref{interp ineq}) seems to be necessary in the proof in \cite{R17}.
Take $w_1 = \eta \bar{u}$, where the smooth cut-off function $\eta$ is given as
\begin{equation*}
\eta(x) =
    \begin{cases}
        1 \text{ for } \quad x\in \mbb{B}^{+}_{1/6},\\
        0 \text{ for } \quad x\in (\mbb{B}^{+}_{1/5})^{c}.\\
    \end{cases}
\end{equation*}
From \eqref{interp ineq}, we  get
\begin{equation}\label{interp pde}
    \|\nabla' (\eta \bar{u}))\|_{L^{2}(\mathbf{B}_{1/5})} \leq \varepsilon \|\nabla^2 (\eta \bar{u})\|_{L^{2}(\mbb{B}^{+}_{1/5})}  + C\varepsilon^{-2} \|\eta \bar{u}\|_{L^{2}(\mathbf{B}_{1/5})}.
\end{equation}
Thus, we have
\begin{equation*}
 \|\nabla' \bar{u} \|_{L^{2}(\mathbf{B}_{1/6})} \leq C\varepsilon( \|\nabla^2 \bar u\|_{L^2(\mathbb B^+_{1/5})}
 +\|\nabla \bar u\|_{L^2(\mathbb B^+_{1/5})}+\|\bar u\|_{L^2(\mathbb B^+_{1/5})})+C\varepsilon^{-2} \|\bar{u}\|_{L^{2}(\mathbf{B}_{1/5})}.
\end{equation*}
Thanks to the estimates (\ref{strong-1}), we have
\begin{equation*}
\begin{aligned}
  \|\nabla \bar{u} \|_{L^{2}(\mbb{B}^{+}_{1/5})}+  \|\nabla^2 \bar{u} \|_{L^{2}(\mbb{B}^{+}_{1/5})} 
    & \leq  e^{C(1+\|H\|^{\frac{2}{3}}_{L^{\infty}} + \|h\|^{2}_{L^{\infty}})}.
    \end{aligned}
\end{equation*}
Then we obtain that
\begin{equation*}
    \|\nabla' \bar{u}\|_{L^{2}(\mathbf{B}_{1/6})} \leq \varepsilon e^{C(1+\|H\|^{\frac{2}{3}}_{L^{\infty}} +\|h\|^{2}_{L^{\infty}})} +C \varepsilon^{-2} \|\bar{u}\|_{L^{2}(\mathbf{B}_{1/5})}.
\end{equation*}
We now add $\|\bar{u}\|_{L^\infty(\mathbf{B}_{1/6})}$ to both sides to get
\begin{equation}
  \|\bar{u}\|_{L^\infty(\mathbf{B}_{1/6})}+  \| \nabla'\bar{u}\|_{L^2(\mathbf{B}_{1/6})} \leq \varepsilon e^{C(1+\|H\|^{\frac{2}{3}}_{L^{\infty}} +\|h\|^{2}_{L^{\infty}})} + C\varepsilon^{-2} \|\bar{u}\|_{L^\infty(\mathbf{B}_{1/5})}.
    \label{lazy-jack}
\end{equation}
We incorporate the first term on the right hand side by choosing $\varepsilon$ such that
$$
\varepsilon e^{C(1+\|H\|^{\frac{2}{3}}_{L^{\infty}} +\|h\|^{2}_{L^{\infty}})} = \frac{1}{2} (\|\bar{u}\|_{L^\infty(\mathbf{B}_{1/6})}+  \| \nabla'\bar{u}\|_{L^2(\mathbf{B}_{1/6})}).
$$
Namely,
\begin{equation*}
    \varepsilon = \bigg( \frac{  \|\bar{u}\|_{L^\infty(\mathbf{B}_{1/6})}+  \| \nabla'\bar{u}\|_{L^2(\mathbf{B}_{1/6})} }{2e^{C(1+\|H\|^{\frac{2}{3}}_{L^{\infty}} +\|h\|^{2}_{L^{\infty}})}}  \bigg).
\end{equation*}
Plugging this in (\ref{lazy-jack}) gives that
\begin{equation*}
   (\|\bar{u}\|_{L^\infty(\mathbf{B}_{1/6})}+  \| \nabla'\bar{u}\|_{L^2(\mathbf{B}_{1/6})})^{3} \leq e^{C(1+\|H\|^{\frac{2}{3}}_{L^{\infty}} +\|h\|^{2}_{L^{\infty}})}\|\bar{u}\|_{L^{\infty}(\mathbf{B}_{1/5})}.
\end{equation*}
Thanks to  \eqref{claim 1}, we have
\begin{equation}\label{claim 22}
    e^{-C(1+\|H\|^{\frac{2}{3}}_{L^{\infty}} + \|h\|^{2}_{L^{\infty}})} \leq \|\bar{u}\|_{L^{\infty}(\mathbf{B}_{1/5})},
\end{equation}
which verifies the claim \eqref{claim 2}.
By the elliptic estimates (\ref{324}), we get
\begin{equation}\label{claim 3}
    \begin{aligned}
        \|\bar{u}\|_{L^{\infty}(\mathbf{B}_{1/4})} & \leq \|\bar{u}\|_{L^{\infty}(\mathbb{B}^+_{1/4})} \\
        &\leq e^{C\|h\|_{L^{\infty}}} \|{u}\|_{L^{\infty}(\mathbb{B}^+_{1/4})} \\
        &  \leq C (\|H\|_\infty^{\frac{n}{2}}+\|h\|_\infty^{n}+1) e^{C\|h\|_\infty} \| u\|_{L^2(\mathbb B^+_{1/2})} \\
        &\leq e^{C(1+\|H\|^{\frac{2}{3}}_{L^{\infty}} + \|h\|^{2}_{L^{\infty}})}\|\bar{u}\|_{L^{2}(\mbb{B}^{+}_{1/2})} \\&=e^{C(1+\|H\|^{\frac{2}{3}}_{L^{\infty}} + \|h\|^{2}_{L^{\infty}})}.
    \end{aligned}
\end{equation}
Combining \eqref{claim 22} and \eqref{claim 3}, we have
\begin{equation}
    \|\bar{u}\|_{L^{\infty}(\mathbf{B}_{1/4})}  \leq e^{C(1+\|H\|^{\frac{2}{3}}_{L^{\infty}} + \|h\|^{2}_{L^{\infty}})}\|\bar{u}\|_{L^{\infty}(\mathbf{B}_{1/5})}.
\end{equation}
Notice that $u = e^{w(x)}\bar{u}$ on $\mathbf{B}_{1/2}$. By rescaling and diffeomorphism of the Fermi exponential map, and (\ref{infty-1}), we arrive at
\begin{equation*}
    \|u\|_{L^{\infty}(\mathbf{B}_{2r}(x_{0}))} \leq e^{C(1+\|H\|^{\frac{2}{3}}_{L^{\infty}} + \|h\|^{2}_{L^{\infty}})}\|u\|_{L^{\infty}(\mathbf{B}_{r}(x_{0}))}
\end{equation*}
for any $x_{0} \in \partial \mathcal{M}$, $\mbb{B}_{2r}(x_{0}) \subset \partial \mathcal{M}$, and $r < r_{0}$ for some $r_{0}$ depending only on $\partial \mathcal{M}$.  Therefore, we have proved Theorem \ref{thm 1}.

\end{proof}

\section{Proof of Carleman estimates}
This section is devoted to the proof of the local Carleman estimates in Proposition \ref{local car} and global Carleman estimates in Proposition \ref{prop GCE}.
\begin{proof}[Proof of Proposition \ref{local car}]
We write the Laplace-Beltrami operator as
\begin{equation*}
    r^2 \Delta_{g} = r^2 \partial^{2}_{r} + r^2 (\partial_{r}\ln\sqrt{\gamma} + \dfrac{n-1}{r})\partial_{r} + \dfrac{1}{\sqrt{\gamma}}\partial_{i}(\sqrt{\gamma}\gamma^{ij}\partial_{j}),
\end{equation*}
where $\partial_{i} = \dfrac{\partial}{\partial\theta_{i}}$, $\gamma_{ij}(r,\theta)$ is a metric on the geodesic sphere $S^{n-1}$, $\gamma^{ij}$ is the inverse of $\gamma_{ij}$, $\gamma=\det (\gamma_{ij})$, and $\theta_{n-1} = \dfrac{x_{n}}{|x|}$.  For $r$ small enough, we have
\begin{equation}\label{r metric}
     \begin{cases} 
      |\partial_{r}\gamma^{ij}| \leq C|\gamma^{ij}|, \\
      |\partial_{r}\gamma| \leq C,\\
      C^{-1} \leq \gamma \leq C,
   \end{cases}
\end{equation}
where $C$ depends on the manifold $\mathcal{M}$. Let $r = e^{t}$. Then $\partial_{r} = e^{-t}\partial_{t}$. Hence the function $u(t,\theta)$ is supported in the half cylinder $\mathcal{N} := S^{n-1}_{+} \times (-\infty,T_{0}]$. We also write $\partial \mathcal{N} = \partial S^{n-1}_{+} \times (-\infty,T_{0}]$. Notice that $T_{0} = ln(r_{0})$ is negative with $T_{0}$ large enough since $r_{0}$ is small. With the new coordinates, we can write
\begin{equation*}
    e^{2t}\Delta = \partial^{2}_{t} +(\partial_{t}\ln\sqrt{\gamma} + n-2)\partial_{t} + \dfrac{1}{\sqrt{\gamma}}\partial_{i}(\sqrt{\gamma}\gamma^{ij}\partial_{j}).
\end{equation*}
The conditions $\eqref{r metric}$ becomes
\begin{equation}\label{t metric}
     \begin{cases} 
      |\partial_{t}\gamma^{ij}| \leq Ce^{t}|\gamma^{ij}|, \\
      |\partial_{t}\gamma| \leq Ce^{t}, \\
      C^{-1} \leq \gamma \leq C.
   \end{cases}
\end{equation} 

 There is no effective way to deal with  the term $hv$ on the boundary for the Carleman estimates. Moreover, $v$ is not in $H^2$. In order to prove \eqref{car est}, we pursue a splitting strategy similar to \cite{GRSU20}, separating the problem into an elliptic equation with solutions in $H^1$ and an elliptic equation with solutions in $H^2$.  To this end, we set $v= v_{1} + v_{2}$ where $v_{1}$ is a solution to
\begin{equation}\label{u1 prob}
\begin{aligned}
        \bigg( \partial^{2}_{t} +(\partial_{t}\ln\sqrt{\gamma} + n-2)\partial_{t} + \dfrac{1}{\sqrt{\gamma}}\partial_{i}(\sqrt{\gamma}\gamma^{ij}\partial_{j}) -K^{2} \tau^{2} \bigg) v_{1}& = e^{2t}Hv+ e^{2t}F \quad in \ \mathcal{N}, \\
         \frac{\partial v_1}{\partial{\nu}} & = e^{t}hv \quad  on \ \partial \mathcal{N},
    \end{aligned}
\end{equation}
and $v_2$ is a solution to
\begin{equation}\label{u2 prob}
\begin{aligned}
        \bigg( \partial^{2}_{t} +(\partial_{t}\ln\sqrt{\gamma} + n-2)\partial_{t} + \dfrac{1}{\sqrt{\gamma}}\partial_{i}(\sqrt{\gamma}\gamma^{ij}\partial_{j})\bigg) v_{2}& =  -K^{2} \tau^{2}v_{1} \hspace{0.5cm}& in \hspace{0.1cm} \mathcal{N}, \\
         \frac{\partial v_2}{\partial{\nu}}& = 0 \hspace{0.5cm}& on \hspace{0.1cm} \partial \mathcal{N},
    \end{aligned}
\end{equation}
where $K >1$ is a large constant to be chosen later.
We now derive separate estimates for $v_{1}$ and $v_{2}$: For $v_{1}$, we use the definition of weak solutions. Moreover, we will make use of the term $K^{2} \tau^{2}v_1$ as the leading term in the power of $\tau$.  For $v_{2}$, we will apply  Carleman estimates. Note that the volume element in polar coordinates is $r^{n-1} \sqrt{\gamma}  dr d\theta$. By replacing $\tau$ by $\tau+\frac{n}{2}$, which only changes the lower bound of $\tau$ by a constant in the Carleman estimates, we can skip the $r^{n-1}$ in the volume element. We split the proof of the Carleman estimates in three steps.

\textit{Step 1:} \textit{Estimate for} $v_{1}$: Because of the term $-K^2\tau^2v_1$ in \eqref{u1 prob}, the solvability of $v_1$  follows from the Lax-Milgram theorem. By the regularity theory of elliptic equations,  the solution $v_1$ is at most in $ H^\frac{3}{2}$.  The weak solution of \eqref{u1 prob} is given as
\begin{equation}\label{u1 weak}
\begin{aligned}
    (\partial_{t}v_{1},\partial_{t} \xi) + (\gamma^{ij}\partial_{j}v_{1},\partial_{i}\xi) - ((\partial_{t}\ln\sqrt{\gamma} + n-2)\partial_{t}v_{1}, \xi) + K^{2}\tau^{2}(v_{1}, \xi)\\
    = -(e^{2t}(Hv+F),\xi) + (e^{t}hv,\xi)_{0} 
    \end{aligned}
\end{equation}
for all $\xi \in H^{1}(\mathcal{N})$, where $(\cdot,\cdot)$ is the $L^{2}$ inner product on $\mathcal{N}$ and $(\cdot,\cdot)_{0}$ is the restriction of the inner product onto $\partial \mathcal{N}$.
We now test \eqref{u1 weak} by taking $\xi =  e^{-2\tau \phi} v_{1}$, where $\phi$ is as defined in (\ref{phi-def}), to get
\begin{align}
    (\partial_{t}v_{1},\partial_{t} (e^{-2\tau \phi} v_{1})) + (\gamma^{ij}\partial_{j}v_{1},\partial_{i}( e^{-2\tau \phi} v_{1})) - ((\partial_{t}\ln\sqrt{\gamma} + n-2)\partial_{t}v_{1},  e^{-2\tau \phi} v_{1}) \nonumber \\ + K^{2}\tau^{2}(v_{1},  e^{-2\tau \phi} v_{1})
    = -(e^{2t}(Hv+F), e^{-2\tau \phi} v_{1}) + (e^{t}hu, e^{-2\tau \phi} v_{1})_{0}.
    \end{align}
We integrate by parts to get
\begin{equation}\label{u1 test}
    \begin{aligned}
        \|e^{-\tau \phi}\partial_{t}v_{1}\|^{2}_{L^{2}} &+\|e^{-\tau \phi}\nabla_{\theta} v_{1} \|^{2}_{L^{2}} + K^{2}\tau^{2}\| e^{-\tau \phi} v_{1}\|^{2}_{L^{2}}  \\
        &- 2\tau(e^{-\tau\phi}\phi'\partial_{t} v_{1},e^{-\tau \phi}v_{1})  \\
        &-(e^{-\tau \phi}(\partial_{t}\ln\sqrt{\gamma}
        + n-2)\partial_{t}v_{1}, e^{-\tau \phi}v_{1})\\
        &= -(e^{-\tau\phi}e^{2t}(Hv+F), e^{-\tau \phi} v_{1})+(e^{-\tau\phi}e^{t} hv, e^{-\tau\phi}v_{1})_{0} 
    \end{aligned}
\end{equation}
where $|\nabla_{\theta}v_{1}|^2 = \gamma^{ij}\partial_{i}v_{1} \partial_{j}v_{1}$.
We  control the inner products in (\ref{u1 test}) by Young's inequality. We estimate the first inner product in (\ref{u1 test}) to obtain  
\begin{equation}\label{inner 1}
    \begin{aligned}
&|2\tau(e^{-\tau\phi}\partial_{t} \phi'v_{1},e^{\tau \phi}v_{1} )| \\
&\leq \frac{1}{8} \|e^{-\tau\phi}\partial_{t}v_{1}\|^{2}_{L^{2}} + 16\tau^{2} \|e^{-\tau \phi} v_{1}\|^{2}_{L^{2}}\\
& \leq \frac{1}{2} \|e^{-\tau\phi}\partial_{t}v_{1}\|^{2}_{L^{2}} + 64\tau^{2} \|e^{-\tau \phi} v_{1}\|^{2}_{L^{2}},
    \end{aligned}
\end{equation}
since $\phi'\approx 1$. By the similar strategy, the second inner product in (\ref{u1 test})  can be controlled by
\begin{equation}\label{inner 2}
    \begin{aligned}
        &|e^{-\tau \phi}(\partial_{t}\ln\sqrt{\gamma} + n-2)\partial_{t}v_{1}, e^{-\tau \phi}v_{1})|  \\
        &\leq \tau^{-2}\|e^{-\tau \phi}(\partial_{t}\ln\sqrt{\gamma} + n-2)\partial_{t}v_{1}\|^{2}_{L^{2}} + \tau^{2}\|e^{-\tau \phi}v_{1}\|^{2}_{L^{2}}\\
&\leq n\tau^{-2}\|e^{-\tau \phi}\partial_{t}v_{1}\|^{2}_{L^{2}} + \tau^{2}\|e^{-\tau \phi}v_{1}\|^{2}_{L^{2}},\\
    \end{aligned}
\end{equation}
where the second inequality follows from \eqref{t metric} as $t$ is negatively large.
The third inner product in (\ref{u1 test}) can be bounded by
\begin{equation}\label{inner 3}
\begin{aligned}
|(e^{-\tau\phi}e^{2t}(Hv+F), e^{-\tau \phi} v_{1})| &\leq \tau^{-2}\|e^{-\tau\phi}e^{2t}(Hv+F)\|_{L^{2}}^{2} + \tau^{2}\|e^{-\tau \phi} v_{1}\|^{2}_{L^{2}}\\
&\leq  \tau^{-2}\|e^{-\tau\phi}e^{2t}(Hv+F)\|_{L^{2}}^{2} + \tau^{2}\|e^{-\tau \phi} v_{1}\|^{2}_{L^{2}}.
\end{aligned}
\end{equation}
There is no effective estimate for the integration on the boundary, due to the merely boundedness of $h$. By the Young's inequality, we get
\begin{equation}\label{3ast}
    |(e^{-\tau\phi} e^{t}hv, e^{-\tau\phi}v_{1})_{0}| \leq C(\varepsilon)\tau^{-1}\|e^{-\tau\phi} e^{t}hv\|^{2}_{0} + \varepsilon \tau \|e^{-\tau\phi}v_{1}\|_{0}^{2}
\end{equation}
where $\varepsilon > 0$ is small to be chosen later and $C(\varepsilon)\approx \frac{C}{\varepsilon}$.
We  further bound these boundary terms by the following trace estimates in \cite{GRSU20},
\begin{equation}\label{trace est Zhu}
\|v\|_{L^{2}(S^{n-1})}^{2} \leq C(\tau \|v\|_{L^{2}(S^{n}_{+})}^{2} + \tau^{-1} \|\nabla_{\theta} v\|_{L^{2}(S^{n}_{+})}^{2}).
\end{equation}
Applying (\ref{trace est Zhu}) to our second boundary term in \eqref{3ast} to have
\begin{equation}
\varepsilon \tau \|e^{-\tau\phi}v_{1}\|_{0}^{2} \leq C\varepsilon( \tau^{2} \|e^{-\tau\phi} e^t v_{1}\|_{L^{2}}^{2} +  \|e^{-\tau\phi}\nabla_{\theta}v_{1}\|_{L^{2}}^{2}).
\end{equation}
Thus, we get 
\begin{equation}\label{inner 4}
\begin{aligned}
|(e^{-\tau\phi} he^{t}v_{1}, e^{-\tau\phi}v_{1})_{0}| &\leq  C(\varepsilon)\tau^{-1}\|e^{-\tau\phi} hv\|^{2}_{0}
 \\
&+ C\varepsilon( \tau^{2} \|e^{-\tau\phi}v_{1}\|_{L^{2}}^{2} +  \|e^{-\tau\phi}\nabla_{\theta}v_{1}\|_{L^{2}}^{2}).
\end{aligned}
\end{equation}
For sufficiently small  $\varepsilon$ and large $K$, combining estimates (\ref{u1 test}), \eqref{inner 1}, \eqref{inner 2}, \eqref{inner 3}, \eqref{inner 4}, and grouping terms to the highest power of $\tau$ gives 
\begin{equation}\label{u1 carleman}
    \begin{aligned}
       & \|e^{-\tau \phi}\partial_{t}v_{1}\|^{2}_{L^{2}} +\|e^{-\tau \phi}\nabla_{\theta} v_{1} \|^{2}_{L^{2}}  + \frac{K^{2}}{2}\tau^{2}\|e^{-\tau \phi} v_{1}\|^{2}_{L^{2}}\\
          &\leq C\tau^{-2}\|e^{-\tau\phi}He^{2t}v\|^{2}_{L^{2}} +C\tau^{-2}\|e^{\tau\phi}e^{2t} F\|^{2}_{L^{2}}+C\tau^{-1}\|e^{-\tau\phi} e^{t} hv\|^{2}_{0},
    \end{aligned}
\end{equation}
where  the last two terms in the right hand side of (\ref{inner 4}) are absorbed into the left hand of (\ref{u1 carleman}).

Next we prove  stronger Carleman estimates with some strong assumption on $v_1$.
Suppose that $\supp v_1\subset \{ x\in \mathbb B^+_{r_0}| r(x)\geq \rho\}$. Let $\hat{T}_0=\ln \rho$. Applying Cauchy-Schwarz inequality, we have
\begin{align}
|\int_{\mathcal{N}}  \partial_t |v_1|^2 e^{-t} e^{-2\tau\phi}\sqrt{\gamma} \, dtd\theta|\leq 2 (\int_{\mathcal{N}}  |\partial_t v_1|^2 e^{-t} e^{-2\tau\phi}\sqrt{\gamma} \, dtd\theta)^{\frac{1}{2}}
(\int_{\mathcal{N}} | v_1|^2 e^{-t} e^{-2\tau\phi}\sqrt{\gamma} \, dtd\theta)^{\frac{1}{2}}.
\label{cauchy-1}
\end{align}
For the left hand side of (\ref{cauchy-1}), integrating by parts shows that
\begin{align}
\int_{\mathcal{N}}  \partial_t |v_1|^2 e^{-t} e^{-2\tau\phi}\sqrt{\gamma} \, dtd\theta&=\int_{\mathcal{N}}  |v_1|^2 e^{-t} e^{-2\tau\phi}\sqrt{\gamma} \, dtd\theta-\int_{\mathcal{N}}  |v_1|^2 e^{-t} e^{-2\tau \phi}\partial_t (\ln\sqrt{\gamma})\sqrt{\gamma} \, dtd\theta \nonumber\\
&+2\tau\int_{\mathcal{N}}  |v_1|^2 e^{-t} e^{-2\tau\phi}\phi'\sqrt{\gamma} \, dtd\theta.
\end{align}
Since $|\partial_t\ln\sqrt{\gamma}|\leq C e^t$ for $|\hat{T}_0|$ large enough and $\phi'\approx 1$, we obtain that
\begin{align}
|\int_{\mathcal{N}}  \partial_t |v_1|^2 e^{-t} e^{-2\tau \phi}\sqrt{\gamma} \, dtd\theta|\geq C \tau\int_{\mathcal{N}}   |v_1|^2 e^{-t} e^{-2\tau \phi}\sqrt{\gamma} \, dtd\theta.
\label{cauchy1-1}
\end{align}
Taking (\ref{cauchy-1}) and (\ref{cauchy1-1}) into consideration gives that
\begin{align}
e^{-\hat{T}_0}\int_{\mathcal{N}}  | \partial_t v_1|^2 e^{-2\tau \phi} \sqrt{\gamma} \, dtd\theta &\geq \int_{\mathcal{N}}  | \partial_t v_1|^2 e^{-t} e^{-2\tau \phi}\sqrt{\gamma} \, dtd\theta \nonumber \\
&\geq C\tau^2\int_{\mathcal{N}}  |v_1|^2 e^{-t} e^{-2\tau \phi}\sqrt{\gamma} \, dtd\theta.
\end{align}
Notice that $e^{-\hat{T}_0}=\rho^{-1}$. Thanks to (\ref{u1 carleman}), we derive that
\begin{equation}
    \begin{aligned}
       & \|e^{-\tau \phi}\partial_{t}v_{1}\|^{2}_{L^{2}} +\|e^{-\tau \phi}\nabla_{\theta} v_{1} \|^{2}_{L^{2}}  + \frac{K^{2}}{2}\tau^{2}\|e^{\tau \phi} v_{1}\|^{2}_{L^{2}}+\tau^2 \rho  \|e^{-\tau \phi} e^{-t/2}v_{1}\|^{2}_{L^{2}}  \\
          &\leq C\tau^{-2}\|e^{-\tau\phi}He^{2t}v\|^{2}_{L^{2}} +C\tau^{-2}\|e^{-\tau\phi}e^{2t} F\|^{2}_{L^{2}}+C\tau^{-1}\|e^{-\tau\phi} e^{t} hv\|^{2}_{0}.
    \end{aligned}
\end{equation}

\textit{Step 2:} \textit{Estimate for} $v_{2}$. We now derive the estimate for $v_{2}$.\\
Let
$$
v_{2} = e^{-\tau \psi} W.
$$
Recall that $\psi(x) = -\phi(\ln(r(x)))$. We introduce the conjugate operator,
\begin{equation*}
\begin{aligned}
\mathcal{L}(W) = r^{2}e^{\tau\psi(x)}\Delta_{g}(e^{-\tau\psi(x)}W) \\
= e^{2t}e^{-\tau\phi(t)}\Delta_{g}(e^{\tau\phi(t)}W). 
\end{aligned}
\end{equation*}
By straightforward calculations, it follows that
\begin{equation*}
    \begin{aligned}
    \mathcal{L}(W)= \partial^{2}_{t}W + (2\tau\phi' + (n-2)+\partial_{t}\ln \sqrt{\gamma})\partial_{t}W \\
    +(\tau^{2}|\phi'|^{2}+ \tau\phi''+ (n-2)\tau\phi'+\tau\partial_{t}\ln \sqrt{\gamma}\phi')W + \Delta_{\theta}W,
    \end{aligned}
\end{equation*}
where
\begin{equation}\label{Lap-Bel}
    \Delta_{\theta}W = \dfrac{1}{\sqrt{\gamma}}\partial_{i}(\sqrt{\gamma}\gamma^{ij}\partial_{j}W)
\end{equation}
is the Laplace-Beltrami operator on $S^{n-1}$.  We introduce the following $L^{2}$ norm
$$
\|W\|_{\phi}^{2} = \int\limits_{\mathcal{N}} |W|^{2} \phi'^{-2} \sqrt{\gamma}dtd\theta,
$$
where $d\theta$ is the measure on $S^{n-1}$ and  $\sqrt{\gamma}$ is the metric on $S^{n-1}$. In the following calculations with the norm $\|W\|_{\phi}$, we will skip the term $\sqrt{\gamma}$ since it will not do any harm because of (\ref{t metric}). By the properties of $\phi$, it is easy to see that this new norm is equivalent to the usual $L^{2}$ norm. To obtain the Carleman estimate, we aim to find a lower bound for $\|\mathcal{L}(W)\|_{\phi}$. By the triangle inequality, we have
$$
\|\mathcal{L}(W)\|_{\phi} \geq\mathcal{B}- \mathcal{A},
$$
where
\begin{equation}\label{A}
\mathcal{B} = \|\partial^{2}_{t}W + 2\tau\phi'\partial_{t}W + \tau^{2}|\phi'|^{2}W  + \Delta_{\theta}W\|_{\phi}
\end{equation}
and
\begin{equation}\label{B}
    \mathcal{A} = \|\tau\phi''W +(\partial_{t}\ln \sqrt{\gamma} + n-2)\tau\phi'W   + (\partial_{t}\ln \sqrt{\gamma} + n-2)\partial_{t}W \|_{\phi}.
\end{equation}
We will show later that $\mathcal{A}$ can be controlled by $\mathcal{B}$. 
Write
\begin{equation}\label{Asq est}
    \mathcal{B}^{2} = \mathcal{B}_{1} + \mathcal{B}_{2} + \mathcal{B}_{3}
\end{equation}
where
\begin{equation}
    \mathcal{B}_{1} =\|\partial^{2}_{t}W  + \tau^{2}|\phi'|^{2}W  + \Delta_{\theta}W\|_{\phi}^{2},
\end{equation}
\begin{equation}\label{A2 est}
    \mathcal{B}_{2} = \|2\tau\phi'\partial_{t}W\|_{\phi}^{2},
\end{equation}
\begin{equation}
    \mathcal{B}_{3} = 2 \langle 2\tau\phi'\partial_{t}W,\partial^{2}_{t}W  + \tau^{2}|\phi'|^{2}W  + \Delta_{\theta}W \rangle_{\phi}.
\end{equation}
We decompose $\mathcal{B}_{3}$ as follows
$$
\mathcal{B}_{3} = J_{1} + J_{2} + J_{3},
$$
where
$$
J_{1} = 4\tau \int\limits_{\mathcal{N}} \phi'^{-1}\partial_{t}W\partial^{2}_{t}W dtd\theta,
$$
$$
J_{2} = 4\tau^{3} \int\limits_{\mathcal{N}} \phi'\partial_{t}W W dtd\theta,
$$
$$
J_{3} = 4\tau \int\limits_{\mathcal{N}} \phi'^{-1}\partial_{t}W\Delta_{\theta}W dtd\theta.
$$
Using $\partial_{t}|\partial_{t}W|^{2} = 2\partial_{t}W\partial_{t}^{2}W$,  for $J_1$, we integrate  by parts with respect to $t$ to get
\begin{equation}\label{K1 est}
    J_{1} = 2 \tau \int\limits_{\mathcal{N}} \phi''\phi'^{-2}|\partial_{t}W|^{2} dtd\theta = -2\tau\||\phi''|^{\frac{1}{2}}\partial_{t}W\|_{\phi}^{2},
\end{equation}
where the last equality follows from the fact that $\phi'' = -|\phi''|$ by the definition of $\phi$ in (\ref{phi-def}).
Similarly, performing integration by parts on $J_{2}$ with respect to $t$ gives that
\begin{equation}\label{K2 est}
\begin{aligned}
        J_{2} &= -2\tau^{3} \int\limits_{\mathcal{N}} \phi''|W|^{2} dtd\theta \\
        &\geq \tau^{3} \||\phi''|^{\frac{1}{2}}W\|_{\phi}^{2},
\end{aligned}
\end{equation}
since $\phi' \approx 1$.

For $J_3$, Integrating by parts  with respect to $\partial_{i}$ show that
\begin{equation*}
\begin{aligned}
        J_{3} = -4\tau \int\limits_{\mathcal{N}} \phi'^{-1}\partial_{t}\partial_{i}W\gamma^{ij}\partial_{j}W dtd\theta + 4\tau \int\limits_{\partial \mathcal{N}} \phi'^{-1}\partial_{t}W\gamma^{j(n-1)}\partial_{j}W dtd\theta. 
\end{aligned}
\end{equation*}
This second integral is $0$ since we have that $\gamma^{j(n-1)}\partial_{j}W = e^{\tau \phi}\frac{\partial v_2}{\partial_{\nu}} = 0$ on $\partial\mathcal{N}$ by the boundary term in \eqref{u2 prob}. 
We integrate by parts with respect to $t$ and use the assumptions for $\gamma^{ij}$ to get
\begin{equation}\label{K3 est}
\begin{aligned}
        J_{3} &= -4\tau \int\limits_{\mathcal{N}} \phi''\phi'^{-2}\partial_{i}W\gamma^{ij}\partial_{j}W dtd\theta + 4\tau \int\limits_{\mathcal{N}} \phi'^{-1}\partial_{i}W\partial_{t}\gamma^{ij}\partial_{j}W dtd\theta \\
        &\geq 3 \tau \int\limits_{\mathcal{N}} |\phi''||\nabla_{\theta}W|^{2}\phi'^{-2} dtd\theta = 3\tau \||\phi''|^{\frac{1}{2}}\nabla_{\theta}W
\|^{2}_{\phi},
\end{aligned}
\end{equation}
where $|\nabla_{\theta} W|^{2} = \gamma^{ij}\partial_{i}W\partial_{j}W$.
Combining the estimates \eqref{K1 est} for $J_{1}$, \eqref{K2 est} for $J_{2}$, and \eqref{K3 est} for $J_{3}$ yields
\begin{equation}\label{A3 est}
    \mathcal{B}_{3} \geq 3\tau \||\phi''|^{\frac{1}{2}}\nabla_{\theta}W\|_{\phi}^{2} + \tau^{3} \||\phi''|^{\frac{1}{2}}W\|_{\phi}^{2} - 2\tau\||\phi''|^{\frac{1}{2}}\partial_{t}W\|_{\phi}^{2}.
\end{equation}
To bound $\mathcal{B}_{1}$ we need to use a stronger norm.  Let $\delta$ to be a small constant which to be determined later.  Since $|\phi''|\leq 1$ and $\tau \geq 1$, we have
\begin{equation}\label{A delt}
    \mathcal{B}_{1} \geq \frac{\delta}{\tau} \Tilde{\mathcal{B}}_{1},
\end{equation}
where $\Tilde{\mathcal{B}}_{1}$ is given by
\begin{equation*}
    \Tilde{\mathcal{B}}_{1} = \||\phi''|^{\frac{1}{2}}(\partial^{2}_{t}W  + \tau^{2}|\phi'|^{2}W  + \Delta_{\theta}W)\|^{2}_{\phi}.
\end{equation*}
We decompose $\Tilde{\mathcal{B}}_{1}$ as
\begin{equation}\label{A tild}
    \Tilde{\mathcal{B}}_{1} = K_{1} + K_{2} + K_{3},
\end{equation}
where
\begin{equation*}
    K_{1} = \||\phi''|^{\frac{1}{2}}(\partial^{2}_{t}W  + \Delta_{\theta}W)\|^{2}_{\phi},
\end{equation*}
\begin{equation*}
    K_{2} = \||\phi''|^{\frac{1}{2}}\tau^{2}|\phi'|^{2}W\|^{2}_{\phi},
\end{equation*}
and
\begin{equation*}
    K_{3} = 2\langle|\phi''|(\partial^{2}_{t}W+ \Delta_{\theta}W),\tau^{2}|\phi'|^{2}W\rangle_{\phi}.
\end{equation*}
We further split $K_{3}$ into $K_{3} = G_{1} + G_{2}$.
Using the fact that $|\phi''| = -\phi''$, we integrate by parts with respect to $t$ to have
\begin{equation*}
\begin{aligned}
    G_{1} &=2\tau^{2}\int\limits_{\mathcal{N}}|\phi''|\partial_{t}^{2}WW dt d\theta \\
    &= 2\tau^{2}\int\limits_{\mathcal{N}}\phi''(\partial_{t}W)^{2} dt d\theta +  2\tau^{2}\int\limits_{\mathcal{N}}\phi'''\partial_{t}WW dt d\theta.
\end{aligned}
\end{equation*}
By Cauchy-Schwarz inequality and properties of $\phi$, we have
\begin{equation}\label{H1}
\begin{aligned}
    G_{1} &\geq -2\tau^{2}\int\limits_{\mathcal{N}}|\phi''|(\partial_{t}W)^{2} dt d\theta -  2\tau^{2}\int\limits_{\mathcal{N}}|\phi'''|((\partial_{t}W)^{2}+W^{2})dt d\theta\\
    &\geq -C\tau^{2}\||\phi''|^{\frac{1}{2}}\partial_{t}W\|_{\phi}^{2} -C\tau^{2}\||\phi''|^{\frac{1}{2}}W\|_{\phi}^{2}.
    \end{aligned}
\end{equation}
For $G_2$, we integrate  by parts with respect to $\partial_{i}$ to get
\begin{equation}\label{H2}
    \begin{aligned}
        G_{2} &=2\tau^{2}\int\limits_{\mathcal{N}}|\phi''|\Delta_{\theta}WW dt d\theta \\
        &= -2\tau^{2}\int\limits_{\mathcal{N}}|\phi''|\gamma^{ij}\partial_{j}W\partial_{i}W dt d\theta + 2\tau^{2}\int\limits_{\partial\mathcal{N}}|\phi''|\gamma^{j(n-1)}\partial_{j}WW dt d\theta \\
        &\geq -2\tau^{2}\||\phi''|^{\frac{1}{2}}\nabla_{\theta}W\|^{2}_{\phi},
    \end{aligned}
\end{equation}
where we used the fact that second integral is $0$ by \eqref{u2 prob}. 
Note that $K_{1} \geq 0$ and $$J_{2} \geq C\tau^{4}\||\phi''|^{\frac{1}{2}}W\|^{2}_{\phi}.$$  It follows from \eqref{A delt}, \eqref{A tild}, \eqref{H1}, and \eqref{H2} that
\begin{equation}\label{A1 est}
    \mathcal{B}_{1} \geq C\tau^{3}\delta\||\phi''|^{\frac{1}{2}}W\|^{2}_{\phi}-C\tau\delta (\||\phi''|^{\frac{1}{2}}\partial_{t}W\|^{2}_{\phi} + \||\phi''|^{\frac{1}{2}}\nabla_{\theta}W\|^{2}_{\phi}).
\end{equation}
Combining \eqref{Asq est}, \eqref{A2 est}, \eqref{A3 est}, and \eqref{A1 est}, we get
\begin{equation*}
    \begin{aligned}
        \mathcal{B}^{2}&\geq C\tau^{3}\delta\||\phi''|^{\frac{1}{2}}W\|^{2}_{\phi} +3\tau \||\phi''|^{\frac{1}{2}}\nabla_{\theta}W\|_{\phi}^{2} + \tau^{3} \||\phi''|^{\frac{1}{2}}W\|_{\phi}^{2} +4\tau^{2}\|\partial_{t}W\|_{\phi}^{2}\\
        &-C\tau\delta \||\phi''|^{\frac{1}{2}}\partial_{t}W\|^{2}_{\phi} -C\tau\delta \||\phi''|^{\frac{1}{2}}\nabla_{\theta}W\|^{2}_{\phi} - 2\tau\||\phi''|^{\frac{1}{2}}\partial_{t}W\|_{\phi}^{2}.
    \end{aligned}
\end{equation*}
Since $\delta$ can be chosen to be small and $|\phi''|$ is small for large $|T_{0}|$, it follows that
\begin{equation}
    \mathcal{B}^{2}\geq C\tau^{3}\||\phi''|^{\frac{1}{2}}W\|^{2}_{\phi} +C\tau \||\phi''|^{\frac{1}{2}}\nabla_{\theta}W\|_{\phi}^{2} +\hat{C}\tau^2\|\partial_{t}W\|_{\phi}^{2},
\end{equation} 
where $\hat{C}$ can be chosen arbitrarily smaller than $C$. If we take $|T_{0}|$ and $\tau$ to be large enough, it is clear that $\mathcal{A}$ is absorbed into $\mathcal{B}$, since the term $(n-2)\tau \phi' W$ in $\mathcal{A}$ can be incorporated initially by the term $\tau^{2}|\phi'|^{2}W$ in $\mathcal{B}$.
Thus,
\begin{equation}
    \|\mathcal{L}(W)\|_{\phi}^{2} \geq C\tau^{3}\||\phi''|^{\frac{1}{2}}W\|^{2}_{\phi} +C\tau \||\phi''|^{\frac{1}{2}}\nabla_{\theta}W\|_{\phi}^{2} +\hat{C}\tau^2\|\partial_{t}W\|_{\phi}^{2}.
    \label{lower}
\end{equation}

Next we prove  stronger Carleman estimates with some strong assumption on $W$.
Suppose that $\supp W\subset \{ y\in \mathbb B^+_{r_0}| r(y)\geq \rho\}$. Let $\hat{T}_0=\ln \rho$. The application of Cauchy-Schwarz inequality gives that
\begin{align}
\int_{\mathcal{N}} \partial_t |W|^2 e^{-t} \sqrt{\gamma} \, dtd\theta\leq 2 (\int_{\mathcal{N}} |\partial_t W|^2 e^{-t} \sqrt{\gamma} \, dtd\theta)^{\frac{1}{2}}
(\int_{\mathcal{N}} | W|^2 e^{-t} \sqrt{\gamma} \, dtd\theta)^{\frac{1}{2}}.
\label{cauchy}
\end{align}
Applying the integration by parts shows that
\begin{align}
\int_{\mathcal{N}} \partial_t |W|^2 e^{-t} \sqrt{\gamma} \, dtd\theta=\int_{\mathcal{N}} |W|^2 e^{-t} \sqrt{\gamma} \, dtd\theta-\int_{\mathcal{N}} |W|^2 e^{-t} \partial_t (\ln\sqrt{\gamma})\sqrt{\gamma} \, dtd\theta.
\end{align}
Thanks to (\ref{t metric}), we have that
\begin{align}
\int_{\mathcal{N}} \partial_t |W|^2 e^{-t} \sqrt{\gamma} \, dtd\theta\geq C \int_{\mathcal{N}}  |W|^2 e^{-t} \sqrt{\gamma} \, dtd\theta.
\label{cauchy1}
\end{align}
Taking (\ref{cauchy}) and (\ref{cauchy1}) into consideration gives that
\begin{align}
e^{-\hat{T}_0}\int_{\mathcal{N}} | \partial_t W|^2  \sqrt{\gamma} \, dtd\theta &\geq \int_{\mathcal{N}} | \partial_t W|^2 e^{-t} \sqrt{\gamma} \, dtd\theta \nonumber \\
&\geq C\int_{\mathcal{N}} |W|^2 e^{-t} \sqrt{\gamma} \, dtd\theta.
\end{align}
Since $e^{-\hat{T}_0}=\rho^{-1}$, from (\ref{lower}),  we arrive at
\begin{equation*}
    \|\mathcal{L}(W)\|_{\phi}^{2} \geq C\tau^{3}\||\phi''|^{\frac{1}{2}}W\|^{2}_{\phi} +C\tau \||\phi''|^{\frac{1}{2}}\nabla_{\theta}W\|_{\phi}^{2} +\hat{C}\tau^2\|\partial_{t}W\|_{\phi}^{2}+C\tau^2\rho\|W e^{-t/2}\|^2_{\phi}.
\end{equation*}
Recall the conjugate operator $v_{2} = e^{\tau\phi}W$. Then we see from the definition of $\mathcal{L}(W)$ that
\begin{equation*}
\begin{aligned}
       \|e^{2t}e^{-\tau\phi}(\Delta v_{2})\|_{\phi}^{2} &\geq C\tau^{3}\||\phi''|^{\frac{1}{2}}e^{-\tau\phi}v_{2}\|^{2}_{\phi} +C\tau \||\phi''|^{\frac{1}{2}}e^{-\tau\phi}\nabla_{\theta}v_{2}\|_{\phi}^{2}\\
    &+\hat{C}\tau\||\phi''|^{\frac{1}{2}}e^{-\tau\phi}\partial_{t}v_{2}\|_{\phi}^{2} - \hat{C}\tau^{3}\||\phi''|^{\frac{1}{2}}e^{-\tau\phi}v_{2}\|_{\phi}^{2}\\
    &+C\tau^2\rho \|e^{-\tau\phi}e^{-t/2}v_{2}\|^{2}_{\phi}.
    \end{aligned} 
\end{equation*}
Since $\hat{C}$ is chosen smaller than $C$, we get
\begin{equation*}
\begin{aligned}
       \|e^{2t}e^{-\tau\phi}(\Delta v_{2})\|_{\phi}^{2} &\geq C\tau^{3}\||\phi''|^{\frac{1}{2}}e^{-\tau\phi}v_{2}\|^{2}_{\phi} +C\tau \||\phi''|^{\frac{1}{2}}e^{-\tau\phi}\nabla_{\theta}v_{2}\|_{\phi}^{2}\\
    &+\hat{C}\tau\||\phi''|^{\frac{1}{2}}e^{-\tau\phi}\partial_{t}v_{2}\|_{\phi}^{2}+C\tau^2\rho \|e^{-\tau\phi}e^{-t/2}v_{2}\|^{2}_{\phi}.
    \end{aligned} 
\end{equation*}
By equivalence of $\| \cdot \|_{\phi}$ and $\|\cdot\|_{L^{2}}$, we arrive at
\begin{equation}\label{u2 carleman}
\begin{aligned}
       \|e^{2t}e^{-\tau\phi}(\Delta v_{2})\|_{L^{2}}^{2} &\geq C\tau^{3}\||\phi''|^{\frac{1}{2}}e^{-\tau\phi}v_{2}\|_{L^{2}}^{2} +C\tau \||\phi''|^{\frac{1}{2}}e^{-\tau\phi}\nabla_{\theta}v_{2}\|_{L^{2}}^{2} \\
    &+\hat{C}\tau\||\phi''|^{\frac{1}{2}}e^{-\tau\phi}\partial_{t}v_{2}\|_{L^{2}}^{2}+C\tau^2\rho \|e^{-\tau\phi}e^{-t/2}v_{2}\|^{2}_{L^{2}}.
    \end{aligned} 
\end{equation}

\textit{Step 3:} \textit{Combining $v_{1}$ and $v_{2}$}: Using \eqref{u1 carleman}, \eqref{u2 carleman}, and the triangle inequality,  we see that
\begin{equation}
    \begin{aligned} \label{weak-jack}
       & \tau\||\phi''|^{\frac{1}{2}}e^{-\tau\phi}\partial_{t}v\|^{2}_{L^{2}}+\tau\||\phi''|^{\frac{1}{2}}e^{-\tau\phi}\nabla_{\theta} v \|^{2}_{L^{2}} + \tau^3\||\phi''|^{\frac{1}{2}} e^{-\tau\phi} v\|^{2}_{L^{2}}\\ &+\tau^2 \rho \|e^{-\tau\phi}e^{-t/2}v\|^{2}_{L^{2}}\\
       & \leq \tau \||\phi''|^{\frac{1}{2}}e^{-\tau\phi}\partial_{t}v_{1}\|^{2}_{L^{2}} +\tau \||\phi''|^{\frac{1}{2}}e^{-\tau\phi}\nabla_{\theta} v_{1} \|^{2}_{L^{2}} + \tau^3\||\phi''|^{\frac{1}{2}} e^{-\tau\phi} v_{1}\|^{2}_{L^{2}}\\&+\tau^2 \rho \|e^{-\tau\phi}e^{-t/2}v_{1}\|^{2}_{L^{2}}\\
        &+ \tau \||\phi''|^{\frac{1}{2}}e^{\tau \psi}\partial_{t}v_{2}\|^{2}_{L^{2}} +\tau \||\phi''|^{\frac{1}{2}}e^{-\tau\phi}\nabla_{\theta} v_{2} \|^{2}_{L^{2}} + \tau^3\||\phi''|^{\frac{1}{2}} e^{-\tau\phi} v_{2}\|^{2}_{L^{2}}\\&+\tau^2 \rho \|e^{-\tau\phi}e^{-t/2}v_{2}\|^{2}_{L^{2}}\\
        &\leq C \tau^{-1}\|e^{-\tau\phi}He^{2t}v\|^{2}_{L^{2}} +C \tau^{-1}\|e^{-\tau\phi}e^{2t} F\|^{2}_{L^{2}}+  C \|e^{2t}e^{-\tau\phi}(\Delta v_{2})\|^2_{L^{2}} \\&+ C \|e^t e^{-\tau\phi} hv\|^{2}_{0}\\
       & \leq C\tau^{-1}\|e^{-\tau\phi}He^{2t}v\|^{2}_{L^{2}} +C\tau^{-1}\|e^{-\tau\phi}e^{2t} F\|^{2}_{L^{2}}+ K^{4}\tau^{4}\|e^{-\tau\phi}v_{1}\|^2_{L^{2}} \\&+C \|e^te^{-\tau\phi} hv\|^{2}_{0}\\
        &\leq C\|e^{-\tau\phi}He^{2t}v\|^{2}_{L^{2}}+C\|e^{-\tau\phi}e^{2t} F\|^{2}_{L^{2}}+C\tau\|e^te^{-\tau\phi} hv\|^{2}_{0},
    \end{aligned}
\end{equation}
where we have used the definition that $e^{2t}\Delta v_{2} = -K^2 \tau^{2} v_{1}$. To control the boundary term, we apply the trace inequality \eqref{trace est Zhu} to have
\begin{equation}
\begin{aligned}\label{english-2}
    \tau \|e^t e^{-\tau\phi} hv\|^{2}_{0}& \leq\tau \|h\|_{L^{\infty}}^{2} \|e^te^{-\tau\phi} v\|_{0}^{2} \\
    &\leq C(\tau^2 \|h\|_{L^{\infty}}^{2} \||\phi''|^{\frac{1}{2}}e^{-\tau\phi} v\|^{2}_{L^{2}} +\|h\|_{L^{\infty}}^2\||\phi''|^{\frac{1}{2}}e^{-\tau\phi} \nabla_{\theta} v\|^{2}_{L^{2}}),
\end{aligned} 
\end{equation}
where we used the fact that $e^t\leq |\phi''|^{\frac{1}{2}}$ as $t$ is negatively large enough.
Letting $\tau\geq C(1+\|h\|_{L^{\infty}}^{2} ) $, we can absorb the terms in (\ref{english-2}) to the left hand side in (\ref{weak-jack}) to get 
\begin{equation}\label{car H}
     \begin{aligned}
        \tau\||\phi''|^{\frac{1}{2}}e^{-\tau\phi}\partial_{t}v\|^{2}_{L^{2}} &+\tau\||\phi''|^{\frac{1}{2}}e^{-\tau\phi}\nabla_{\theta} v \|^{2}_{L^{2}} + \tau^3\||\phi''|^{\frac{1}{2}} e^{-\tau\phi} v\|^{2}_{L^{2}}+\tau^2 \rho \|e^{-\tau\phi}e^{-t/2}v\|^{2} \\
&\leq C \|e^{-\tau\phi}He^{2t}v\|^{2}_{L^{2}} +C\|e^{-\tau\phi}e^{2t} F\|^{2}_{L^{2}}.
\end{aligned}
\end{equation}
By choosing
\begin{align}
    \tau \geq (1 + \|H\|^{\frac{2}{3}}_{L^{\infty}}+\|h\|_{L^{\infty}}^{2} ),
\end{align}
we have
\begin{equation}
     \begin{aligned}
        \tau\||\phi''|^{\frac{1}{2}}e^{-\tau\phi}\partial_{t}v\|^{2}_{L^{2}} &+\tau\||\phi''|^{\frac{1}{2}}e^{-\tau\phi}\nabla_{\theta} v \|^{2}_{L^{2}} + \tau^3\||\phi''|^{\frac{1}{2}} e^{-\tau\phi} v\|^{2}_{L^{2}}+\tau^2 \rho \|e^{-\tau\phi}e^{-t/2}v\|^{2} \\
&\leq C\|e^{-\tau\phi}e^{2t} F\|^{2}_{L^{2}}.
\end{aligned}
\end{equation}
 By  converting to $x$ variables, we arrive at the desired Carleman estimates.
\end{proof}

Next we briefly present the proof for the global Carleman estimates in Proposition \ref{prop GCE}.
We rely on a global Carleman estimates for strong solutions in $H^2$. Since $v$ is $H^\frac{3}{2}$ in (\ref{des-1}), we show the argument by approximation.

\begin{proof}[Proof of Proposition \ref{prop GCE}] 
We consider $v$ in (\ref{des-1}) with $\supp v\subset \mathbb{B}_{2}^{+}$.
For $F\in L^2(\mathbb{B}_{2}^{+})$, and $g\in L^2(\partial\mathbb{B}_{2}^{+}\cap \{x  |  x_{n} = 0\}),$ by the regularity result in Lemma \ref{lem-re}, it holds that
\begin{align*}
   \|v\|_{H^\frac{3}{2}(\mathbb{B}_{2}^{+})} \leq  C (\|F\|_{L^2( \mathbb{B}_{2}^{+})}+ \|g\|_{L^2(\partial\mathbb{B}^+_2\cap\{x|x_n=0\})}).
\end{align*}

We consider a sequence of $F_k\in C_0^\infty (\mathbb{B}^+_2)$ such that $F_k\to F$ in $L^2(\mathbb{B}^+_2)$ and  a sequence of $g_k\in C_0^\infty (\partial\mathbb{B}^+_2\cap\{x|x_n=0\})$   such that $g_k\to g$ in $L^2(\partial\mathbb{B}^+_2\cap\{x|x_n=0\})$. Hence there exists $v_k$ such that 
 \begin{equation*}
    \begin{cases}
        \Delta_{g} v_k = F_k&  \hspace{0.5cm} in \hspace{0.1cm} \mathbb{B}_{2}^{+}, \\
         \frac{\partial v_k} {\partial {\nu}} = g_k& \hspace{0.5cm} on \hspace{0.1cm} \partial\mathbb{B}_{2}^{+}\cap \{x  |  x_{n} = 0\}.  
    \end{cases}
\end{equation*}
It holds that $v_k\in C_0^\infty(\mathbb{B}_{2}^{+})$. By the Carleman estimates in \cite{Z23} (see Proposition 3) or \cite{LR95}, we have
\begin{equation}\label{carleman-1}
\begin{aligned}
     &\|e^{\tau\psi} F_k\|_{L^{2}(\mbb{B}^{+}_{2})}+ \tau^{\frac{3}{2}}s^{2}\|\psi^{\frac{3}{2}}e^{\tau\psi}v_k\|_{L^{2}(\partial\mbb{B}^{+}_{2}\cap \{x  |  x_{n} = 0\})} + \tau^{\frac{1}{2}}\|e^{\tau\psi}g_k\|_{L^{2}(\partial\mbb{B}^{+}_{2}\cap \{x  |  x_{n} = 0\})} \\
   & + \tau^{\frac{1}{2}}s\|\psi^{\frac{1}{2}}e^{\tau\psi}\nabla' v_k\|_{L^{2}(\partial\mbb{B}^{+}_{2}\cap \{x  |  x_{n} = 0\})}
    \\
    &\geq C_{0}\tau^{\frac{3}{2}}s^{2}\|\psi^{\frac{3}{2}}e^{\tau\psi}v_k\|_{L^{2}(\mbb{B}^{+}_{2})} + C_{0}\tau^{\frac{1}{2}}s\|\psi^{\frac{1}{2}}e^{\tau\psi}\nabla v_k\|_{L^{2}(\mbb{B}^{+}_{2})}.
    \end{aligned}
\end{equation}
It is true that
\begin{align}
  \|v_k-v\|_{H^\frac{3}{2}(\mbb{B}^{+}_{2})} \leq  C (\|F-F_k\|_{L^2( \mbb{B}^{+}_{2})}+ \|g-g_k\|_{L^2(\partial \mbb{B}^{+}_{2}\cap \{x  |  x_{n} = 0\} )}). 
  \label{limit}
\end{align}
Thanks to (\ref{limit}), by the trace inequality (\ref{trace}) and smoothness of $\psi$ in  $\mbb{B}^{+}_{2}$, we can pass the limit in (\ref{carleman-1}). Thus, we obtain the Carleman estimate (\ref{carleman}).
\end{proof}

\section{Appendix}
In this appendix, we present the results of a quantitative Caccioppoli inequality and a $H^{\frac{3}{2}}$ regularity estimate for elliptic equations with Neumann boundary conditions. 
We get rid of  the boundary term  in the following quantitative Caccioppoli inequality.
\begin{lemma}\label{corollary Cac}
    Suppose $u$ is a solution of \eqref{bound PDE} and $\cball{r}^{+} \subset \cball{R}^{+} \subset \cball{1/2}^{+}$. Then
\begin{equation}\label{caccioppoli no bdry}
    \|\nabla u \|_{L^{2}(\mbb{B}^+_{r})}  \leq  \frac{C(1+\|H\|^{\frac{1}{2}}_{L^{\infty}} + \|h\|_{L^{\infty}})}{R-r}\|u\|_{L^{2}(\mbb{B}^{+}_{R})},
\end{equation}
where $C = C(g)$.
\end{lemma}

\begin{proof}

Take $\phi = \eta^{2} u$ where $\eta = 0$ in $\cball{1/2}^{+} \setminus \cball{R}^{+}$, $\eta = 1$ in $\cball{r}^{+}$ and $|\nabla \eta| \leq \frac{C}{R-r}$. 
  Since $u$ is a solution of \eqref{bound PDE}, we can multiply \eqref{bound PDE} by $\phi$ and integrate over $\cball{1/2}^{+}$ to get
  $$
2\int_{\cball{1/2}^{+}} g^{ij}D_{j}u D_{i}\eta u \eta dx + \int_{\cball{1/2}^{+}} g^{ij}D_{j}u D_{i}u \eta^2 dx = \int_{\partial \cball{1/2}^{+}\cap\{x|x_n=0\}}  \frac{\partial u}{\partial \nu}\eta^{2} u dS  + \int_{\cball{1/2}^{+}}Hu^2 \eta^{2}  dx.
$$
By the properties of $\eta$ and $g$, we can see that
\begin{equation}\label{cac p1}
2 \int_{\cball{R}^{+}} g^{ij}D_{j}u D_{i}\eta u \eta dx + \Lambda^{-1}\int_{\cball{R}^{+}} |\nabla u|^2 \eta^{2} dx \leq  \int_{\lball{R}} hu^2 \eta^2  dS  + \int_{\cball{R}^{+}}Hu^2\eta^2  dx,
\end{equation}
where $\lball{R}\subset \partial \cball{1/2}^{+}\cap\{x|x_n=0\}$ and $\Lambda^{-1} |\xi|^2\leq g^{ij} \xi_i \xi_j$. By Cauchy-Schwartz inequality and Young's inequality, we can see that
\begin{equation}\label{cac p2}
 \int_{\cball{R}^{+}}g^{ij} D_{j}u D_{i}\eta u \eta dx \leq \varepsilon \int_{\cball{R}^{+}} |\nabla u|^{2} \eta^{2} dx + \frac{C}{\varepsilon} \int_{\cball{R}^{+}} |\nabla \eta|^{2} u^{2} dx
\end{equation}
for any $\varepsilon > 0$. Combining \eqref{cac p1} and \eqref{cac p2}, and taking $\varepsilon = \frac{1}{4\Lambda}$, we get
\begin{equation}\label{cac p3}
\int_{\cball{R}^{+}} |\nabla u|^2 \eta^{2} dx \leq  C\bigg(\int_{\cball{R}^{+}} |\nabla \eta|^{2} u^{2} dx + \|h\|_{L^{\infty}}\int_{\lball{R}} u^2 \eta^2  dS  + \|H\|_{L^{\infty}}\int_{\cball{R}^{+}}u^2\eta^2  dx\bigg).
\end{equation}
Applying the trace inequality (\ref{trace est Zhu})  to the boundary term gives that
\begin{align*}
 \|h\|_{L^{\infty}}\int_{\lball{R}} u^2 \eta^2  dS\leq   \|h\|_{L^{\infty}} \big(\alpha \int_{\cball{R}^{+}}u^2\eta^2  dx + 2\alpha^{-1}(\int_{\cball{R}^{+}}|\nabla u| u\eta^2  dx + \int_{\cball{R}^{+}}u^2 |\nabla \eta|\eta dx) \big)
\end{align*}
for any $\alpha>0$.
Taking $\alpha = \|h\|_{L^{\infty}}$ and applying Young's inequality, we have
\begin{equation}\label{cac p4}
\begin{aligned}
\int_{\cball{R}^{+}} |\nabla u|^2 \eta^{2} dx &\leq  C\bigg(\int_{\cball{R}^{+}} |\nabla \eta|^{2} u^{2} dx +  \|H\|_{L^{\infty}}\int_{\cball{R}^{+}}u^2\eta^2  dx \\
&+(\|h\|_{L^{\infty}}^2+1) \int_{\cball{R}^{+}}  \eta^{2} u^{2} dx
\bigg).
\end{aligned}
\end{equation}
It follows from (\ref{cac p3}), and  the properties of $\eta$ and $\nabla \eta$ that
$$
\int_{\cball{r}^{+}} |\nabla u|^2  dx \leq  C\bigg(\frac{1}{(R-r)^{2}}\int_{\cball{R}^{+}} u^{2} dx + (\|h\|^2_{L^{\infty}}+1)\int_{\cball{R}^+} u^2  dS  + \|H\|_{L^{\infty}}\int_{\cball{R}^{+}}u^2  dx\bigg).
$$
Thus, we arrive at the desired results in the lemma.
\end{proof}
The next lemma is on the regularity of the elliptic equation with Neumann boundary conditions
\begin{equation}
     \begin{cases} 
     \label{regu-n}
      \Delta u = f \quad & x\in \Omega, \\
      \frac{\partial u}{\partial \nu} = g & x\in \partial\Omega,
   \end{cases}
\end{equation}
where $f\in L^2(\Omega)$ and $g\in L^2(\partial\Omega)$ with compatible  condition $\int_{\Omega} f=\int_{\partial\Omega} g $,
and $\Omega$ is a $C^2$ domain. Actually we can assume $\Omega$ is a Lipschtiz domain in the proof below. We can not find the exact literature for the following results which should be known by experts. We briefly show its proof for the completeness of presentation.
\begin{lemma}\label{lem-re}
   There exists a unique solution $u$ in (\ref{regu-n}), up to an additive constant, such that
\begin{align}
    \|u\|_{H^\frac{3}{2}(\Omega)}\leq C (\|f\|_{L^2( \Omega)}+ \|g\|_{L^2(\partial \Omega)}).
    \label{fff}
\end{align}
\end{lemma}
\begin{proof}
We first consider the Neumann boundary value problem
\begin{equation*}
     \begin{cases} 
      \Delta v = 0 \quad & x\in \Omega, \\
      \frac{\partial v}{\partial \nu} = g_1 & x\in \partial\Omega
   \end{cases}
\end{equation*}
with the compatible continuation $\int_{\partial \Omega} g_1=0$. Up to an additive constant, it holds that in \cite{JK81} that
\begin{align}
    \|v\|_{H^\frac{3}{2}(\Omega)}\leq C \|g_1\|_{L^2(\partial \Omega)}.
    \label{j-k}
\end{align}

 We extend $f$ to be a function $f_0$ in a large smooth region $\Omega_0$ such that $\Omega\Subset \Omega_0$, $f_0=f$ in $\Omega$, $\|f_0\|_{L^2(\Omega_0)}\leq C\|f\|_{L^2(\Omega)}$. By the standard theory from elliptic equation, there exists a unique solution $F_0$ for the following equation 
\begin{equation*}
     \begin{cases} 
      \Delta F_0 = f_0 \quad & x\in \Omega_0, \\
      F_0 = 0 & x\in \partial\Omega_0,
   \end{cases}
\end{equation*}
and $\|F_0\|_{H^2(\Omega_0)}\leq C\|f_0\|_{L^2(\Omega)}\leq C\|f\|_{L^2(\Omega)}$. Then $u-F_0$ satisfies the following equation 
\begin{equation*}
     \begin{cases} 
      \Delta (u-F_0 )= 0 \quad & x\in \Omega, \\
      \frac{\partial (u-F_0)}{\partial \nu} = g-\frac{\partial F_0}{\partial \nu} & x\in \partial\Omega.
   \end{cases}
\end{equation*}
The standard trace inequality holds for any $f_1\in H^\frac{3}{2}(\Omega) $,
\begin{align}
    \|f_1\|_{H^1(\partial \Omega)}\leq \|f_1\|_{H^\frac{3}{2}(\Omega)}.
    \label{trace}
\end{align}
It follows (\ref{j-k}) and (\ref{trace}) that
\begin{align*}
   \|u\|_{H^\frac{3}{2}(\Omega)}&\leq  \|F_0\|_{H^\frac{3}{2}( \Omega)}+ C\| g-\frac{\partial F_0}{\partial \nu}\|_{L^2(\partial \Omega)}\nonumber\\
   &\leq  C\|F_0\|_{H^\frac{3}{2}( \Omega)}+ C\| g\|_{L^2(\partial \Omega)}) \leq  C (\|f\|_{L^2( \Omega)}+ \|g\|_{L^2(\partial \Omega)}).
\end{align*}

\end{proof}

\bibliography{Mybib}
\bibliographystyle{abbrv}

\end{document}